\documentclass[a4paper]{article}

\usepackage[latin1]{inputenc} 
\usepackage{graphicx}
\usepackage[pdftex,bookmarks]{hyperref}
\usepackage{amsmath}    
\usepackage{amssymb}
\usepackage{amsmath,amsthm}
\usepackage{textcomp}
\usepackage[all]{xy}
\usepackage{geometry}

\newtheorem{teo}{Theorem}[section]
\newtheorem{defi}{Definition}[section]

\newtheorem{cor}{Corollary}[section]
\newtheorem{prop}{Proposition}[section]

\newtheorem{rem} {Remark}[section]

\newcommand{\nn}{\nonumber}

\DeclareMathOperator{\im}{im}

\DeclareMathOperator{\dvol}{dvol}

\DeclareMathOperator{\supp}{supp}

\DeclareMathOperator{\vol}{vol}
\DeclareMathOperator{\reg}{reg}
\DeclareMathOperator{\sing}{sing}

\title{\huge \bf Symplectic manifolds, $L^p$-cohomology and $q$-parabolicity}  
\author{Francesco Bei  \bigskip \\
Dipartimento di matematica, universit\`a degli studi di Padova\\ E-mail addresses: bei@math.unipd.it\ \     francescobei27@gmail.com }
\date{}

\begin{document}

\maketitle
\begin{abstract}
\noindent Let $(M,\omega,J,g)$ be a non-compact almost K\"ahler manifold. In this paper we provide various criteria that assure that $\omega^k$ induces a non trivial class in the reduced $L^p$ maximal/minimal cohomology of $(M,g)$. Furthermore in the last part we explore some topological applications of our results.
\end{abstract}
\vspace{1 cm}

\vspace{0.5 cm}



\section*{Introduction}
Let $(N,\omega)$ be a compact symplectic manifold of dimension $2n$.  Among its basic properties there is the well known fact that $\omega^k$ induces a non trivial class in the de Rham cohomology of $N$, that is  
\begin{equation}
\label{nonsvanisce}
0\neq [\omega^k]\in H^{2k}_{\text{dR}}(N)
\end{equation}
for any $k=1,...,n$. Consider now a non-compact symplectic manifold $(M,\omega)$. Let $J$ be an almost complex structure compatible with $\omega$ and let $g$ be the Riemannian metric induced by $\omega$ and $J$, that is $g(X,Y):=\omega(X,JY)$ for any $X,Y\in \mathfrak{X}(M)$. Usually in the literature a manifold $M$ equipped with three tensors $\omega$, $J$ and $g$ as above  is called an {\em almost K\"ahler manifold}, see e.g. \cite{ToWeYa}. For various reasons, besides the usual de Rham cohomology, in the non-compact setting it is also interesting to consider the $L^p$-de Rham cohomology which, roughly speaking,  is defined as the quotient between $L^p$-closed forms modulo $L^p$-exact forms \footnote{We refer to the first section of this paper for a precise definition.}. Thus, looking at \eqref{nonsvanisce}, it is natural to wonder whether something similar holds true also for the $L^p$-cohomology of a non-compact almost K\"ahler manifold.  Certainly we cannot expect that a generalization of \eqref{nonsvanisce} holds true for an arbitrary non-compact almost K\"ahler manifold  without any further assumption on $(M,\omega,J,g)$ or $p$. First, as $\omega$ is parallel with respect to some Hermitian connection, it is necessary to assume that $\vol_g(M)<\infty$ in order to have $\omega^k\in L^p\Omega^{2k}(M,g)$ and $p\in [1,\infty)$. Secondly there are celebrated vanishing theorems for the $L^2$-cohomology of certain complete K\"ahler manifolds $(Z,h)$ based on the fact that the corresponding K\"ahler form $\omega$ admits a primitive in $L^{\infty}\Omega^1(Z,h)$, that is there exists $\eta \in \Omega^1(Z)\cap L^{\infty}\Omega^1(Z,h)$ such that $d_1\eta=\omega$, see e.g. \cite{Grom} and \cite{ToH}. Thus we can say that the aim of this paper is to provide an answer to the following question:
\begin{itemize}
\item Let $(M,\omega,J,g)$ be a non-compact almost K\"ahler manifold. Under what circumstances does $\omega^k$ induce a non-trivial class in the $L^p$-cohomology of $(M,g)$?
\end{itemize}
We have already mentioned above that we are led  to assume $\vol_g(M)<\infty$. In what follows we will see that another important property is the $q$-parabolicity (and other criteria inspired by the notion of $q$-parabolicity) of $(M,g)$.
Let us now provide some more details by explaining the structure of this paper. The first section is devoted to the background material about $L^p$-cohomology and almost K\"ahler manifolds. The second section is split in two parts. Its first subsection collects various technical propositions that will be needed later.  The second subsection contains the main results of this paper. More precisely it is devoted to various criteria assuring the non-vanishing of $[\omega^k]$ in $\overline{H}^{2k}_{p,\max/\min}(M,g)$, where $\overline{H}^{2k}_{p,\max/\min}(M,g)$ is the reduced $L^p$-maximal/minimal cohomology of $(M,g)$ respectively, see \eqref{gintonic} and \eqref{moscowmule}.
Concerning the non-vanishing of $[\omega^k]$ in $\overline{H}^{2k}_{p,\max}(M,g)$ we prove the following

\begin{teo}
\label{scalamari}
Let $(M,\omega,J,g)$ be a possibly incomplete almost K\"ahler manifold of finite volume and dimension $2m>0$.  Assume that $\overline{H}^{2m}_{p,\max}(M,g)=\overline{H}^{2m}_{p,\min}(M,g)$ for some $1\leq p<\infty$. Then $\omega^k$ induces a non trivial class in $\overline{H}^{2k}_{q,\max}(M,g)$ for any $k=1,2,...,m$ and $q\in [p,\infty]$.
\end{teo}

Requiring $(M,g)$ to be $q$-parabolic, see Def. \ref{parax}, we can also deduce  the non-vanishing of  $[\omega^k]$ in $\overline{H}^{2k}_{p,\min}(M,g)$ for certain $p$. More precisely we have 

\begin{teo}
\label{sqpara}
Let $(M,\omega,J,g)$ be a possibly incomplete almost K\"ahler manifold of finite volume and dimension $2m>0$. Assume that $(M,g)$ is $p$-parabolic with $1<p<\infty$ and let $p'=p/(p-1)$. Then:
\begin{enumerate}
\item For any $k=1,2,...,m$, $\omega^k$ induces a non trivial class in    $\overline{H}^{2k}_{q,\max}(M,g)$ for any  $q\in [p',\infty]$.
\item For any $k=1,2,...,m$, $\omega^k$ induces a non trivial class in    $\overline{H}^{2k}_{q,\min}(M,g)$ for any  $q\in [1,p]$.
\end{enumerate}
\end{teo}

Finally the last section contains various examples and applications. We exhibit, especially in the framework of K\"ahler manifolds, various examples of metrics satisfying the above theorems. Moreover we show some topological applications of our results. In particular we prove the following

\begin{teo}
\label{ssinghom}
Let $(X,h)$ be a compact and irreducible K\"ahler space of complex dimension $m$. Assume that every point $p\in \sing(X)$ has a local base of open neighborhoods whose regular parts are connected. Then  $$H^{2k}(X,\mathbb{R})\neq  \{0\}$$ for each $k=0,1...,m$. In particular if $X$ is a compact and irreducible normal K\"ahler space then $$H^{2k}(X,\mathbb{R})\neq  \{0\}$$ for each $k=0,...,m$.
\end{teo}

\vspace{1 cm}

\noindent\textbf{Acknowledgments.}  
This work was performed within the framework of the LABEX MILYON (ANR-10-LABX-0070) of Universit\'e de Lyon, within the program ''Investissements d'Avenir'' (ANR--11--IDEX--0007) operated by the French National Research Agency (ANR). It is a pleasure to thank Markus Banagl, Jacopo Gandini, Luca Migliorini and Jon Woolf  for helpful conversations.
 
\section{Background material}

Before starting we point out that in this paper a manifold $M$ will be always assumed to be {\em connected}. The aim of this section is to  recall briefly some basic notions about $L^p$-spaces and $L^p$-cohomology.  We refer to \cite{GoTro}, \cite{GoTr}, \cite{PiRiSe} and \cite{BoY} for proofs and more details. Let $(M,g)$  be an open and possibly incomplete Riemannian manifold of dimension $m$ and let $\dvol_g$ be the one-density associated to $g$.  We consider $M$ endowed with the corresponding {\em Riemannian measure}, see for instance \cite{GYA} pag. 59 or \cite{BGV} pag. 29.  A $k$-form $\omega$ is said {\em measurable} if, for any trivialization $(U,\phi)$ of $\Lambda^kT^*M$, $\phi(\omega|_U)$ is given by a $k$-tuple of measurable functions.  Given a measurable $k$-form $\omega$ the pointwise norm $|\omega|_g$ is defined as $|\omega|_g:= (g(\omega,\omega))^{1/2}$, where with a little abuse of notation we still label by $g$ the metric induced by $g$ on $\Lambda^kT^*M$. Then for every $p\in [1,\infty)$  we can define $L^{p}\Omega^k(M,g)$ as the  space  of measurable  $k$-forms such that    $$\|\omega\|_{L^{p}\Omega^k(M,g)}:=\left(\int_{M}|\omega|_g^p\dvol_g\right)^{1/p}<\infty.$$
For each $p\in [1, \infty)$ we have a Banach space,  for each $p\in (1, \infty)$  we get a reflexive Banach space and  in the case $p=2$ we have a Hilbert space whose inner product is $$\langle \omega, \eta \rangle_{L^2\Omega^k(M,g)}:= \int_{M}g(\omega,\eta)\dvol_g.$$ Moreover $\Omega^k_c(M)$,  the space of smooth $k$-forms with compact support,  is dense in $L^p\Omega^k(M,g)$ for $p\in [1,\infty).$ Finally $L^{\infty}\Omega^{k}(M,g)$ is defined as the space of measurable $k$-forms whose {\em essential\ supp} is bounded, that is the space of measurable  $k$-forms such that $|\omega|_g$ is bounded almost everywhere. Also in this case we get a Banach space.\\
 Consider now the de Rham differential $d_k: \Omega^k_c(M)\rightarrow \Omega_c^{k+1}(M)$ and let $d^t_k: \Omega^{k+1}_c(M)\longrightarrow \Omega^k_c(M)$ be its  {\em formal adjoint}. As it is well known $d^t_k$ is the differential operator uniquely characterized by the following property: for each $\omega\in\Omega_c^k(M)$ and for each $\eta\in \Omega^{k+1}_c(M)$ we have  $$\int_{M}g(\omega,d_k^t\eta)\dvol_g=\int_Mg(d_k\omega,\eta)\dvol_g.$$ For any $p\in [1,\infty)$ we can look at $d_k$ as an unbounded, densely defined and closable operator  acting between $L^p\Omega^k(M,g)$ and $L^p\Omega^{k+1}(M,g)$. In general $d_k$ admits several closed extensions. For our purposes we recall now the definitions of the maximal and minimal one.
The domain of the {\em maximal extension} of $d_k:L^p\Omega^k(M,g)\longrightarrow L^p\Omega^{k+1}(M,g)$ is defined as\\ 
\begin{align}
\label{ner}
& \mathcal{D}(d_{k,\max,p}):=\{\omega\in L^{p}\Omega^k(M,g): \text{there is}\ \eta\in L^p\Omega^{k+1}(M,g)\ \text{such that}\ \int_{M}g(\omega,d_k^t\phi)\dvol_g=\\
 \nn &=\int_{M}g(\eta,\phi)\dvol_g\ \text{for each}\ \phi\in \Omega^{k+1}_c(M)\}.\ \text{In this case we put}\ d_{k,\max,p}\omega=\eta.
\end{align} In other words the maximal extension of $d_k$ is the one defined in the {\em distributional sense}. The domain of the {\em minimal extension} of $d_k:L^p\Omega^k(M,g)\longrightarrow L^p\Omega^{k+1}(M,g)$ is defined as\\ 
\begin{align}
\label{spinaci}
& \mathcal{D}(d_{k,\min,p}):=\{\omega\in L^p\Omega^k(M,g)\ \text{such that there is a sequence}\ \{\omega_i\}\in \Omega_c^k(M)\ \text{with}\ \omega_i\rightarrow \omega\\ \nn & \text{in}\ L^{p}\Omega^k(M,g)\ \text{and}\ d_k\omega_i\rightarrow \eta\ \text{in}\ L^p\Omega^{k+1}(M,g)\ \text{to some }\ \eta\in L^p\Omega^{k+1}(M,g)\}.\ \text{We put}\ d_{k,\min,p}\omega=\eta.
\end{align} Briefly the minimal extension of $d_k$ is the closure of $\Omega^{k}_c(M)$ under the graph norm of $d_k$. Clearly $\mathcal{D}(d_{k,\min,p})\subset \mathcal{D}(d_{k,\max,p})$ and $d_{k,\max,p}\omega=d_{k,\min,p}\omega$ for any $\omega\in \mathcal{D}(d_{k,\min,p})$. Note that for $k=0$ the maximal domain $\mathcal{D}(d_{0,\max,p})$ is nothing but the {\em Sobolev space} $W^{1,p}(M,g)$ while the minimal domain $\mathcal{D}(d_{0,\min,p})$ is nothing but the Sobolev space $W^{1,p}_0(M,g)$, that is the closure of $C^{\infty}_c(M)$ in $W^{1,p}(M,g)$. \\ Also in the case $p=\infty$ $d_k$ admits a closed extension $d_{k,\max,\infty}:L^{\infty}\Omega^k(M,g)\longrightarrow L^{\infty}\Omega^{k+1}(M,g)$ defined as\\ 
\begin{align}
\label{nerop}
& \mathcal{D}(d_{k,\max,\infty}):=\{\omega\in L^{\infty}\Omega^k(M,g): \text{there is}\ \eta\in L^{\infty}\Omega^{k+1}(M,g)\ \text{such that}\ \int_{M}g(\omega,d_k^t\phi)\dvol_g=\\
 \nn &=\int_{M}g(\eta,\phi)\dvol_g\ \text{for each}\ \phi\in \Omega^{k+1}_c(M)\}.\ \text{In this case we put}\ d_{k,\max,\infty}\omega=\eta.
\end{align} In other words a measurable $k$-form $\omega\in L^{\infty}\Omega^k(M,g)$ lies in $\mathcal{D}(d_{k,\max,\infty})$ if the distributional action of $d_k$ applied to $\omega$ lies in $L^{\infty}\Omega^{k+1}(M,g)$.\\
It is easy to verify that if $\omega\in \mathcal{D}(d_{k,\max/\min,p})$ then $d_{k,\max/\min,p}\omega\in \mathcal{D}(d_{k+1,\max/\min,p})$ and the corresponding compositions are identically zero, that is $d_{k+1,\max,p}\circ d_{k,\max,p}\equiv 0$ and $d_{k+1,\min,p}\circ d_{k,\min,p}\equiv 0$. Analogously $d_{k+1,\max,\infty}\circ d_{k,\max,\infty}\equiv 0$ on $\mathcal{D}(d_{k,\max,\infty})$. The $L^p$-maximal/minimal de Rham cohomology of $(M,g)$ is defined as $$H^k_{p,\max/\min}(M,g):=\ker(d_{k,\max,p})/\im(d_{k-1,\max/\min,p})$$ while the reduced $L^p$-maximal/minimal de Rham cohomology of $(M,g)$ is defined as 
\begin{equation}
\label{gintonic}
\overline{H}^k_{p,\max/\min}(M,g):=\ker(d_{k,\max,p})/\overline{\im(d_{k-1,\max/\min,p})},
\end{equation}
where $\overline{\im(d_{k-1,\max/\min,p})}$ is the closure of $\im(d_{k-1,\max/\min,p})$ in $L^p\Omega^k(M,g)$ respectively.
Clearly the identity $\ker(d_{k,\max/\min,p})\rightarrow \ker(d_{k,\max/\min,p})$ induces a surjective map 
\begin{equation}
\label{surj1}
H^k_{p,\max/\min}(M,g)\rightarrow \overline{H}^k_{p,\max/\min}(M,g).
\end{equation}
Similarly we have the $L^{\infty}$-de Rham cohomology and the reduced $L^{\infty}$-de Rham cohomology defined respectively as $$H^k_{\infty,\max}(M,g):=\ker(d_{k,\max,\infty})/\im(d_{k-1,\max,\infty})$$ and 
\begin{equation}
\label{moscowmule}
\overline{H}^k_{\infty,\max}(M,g):=\ker(d_{k,\max,\infty})/\overline{\im(d_{k-1,\max,\infty})}.
\end{equation}
 Clearly also in this case we have a surjective map \begin{equation}
\label{surj2} 
H^k_{\infty,\max}(M,g)\rightarrow \overline{H}^k_{\infty,\max}(M,g)
\end{equation}
 induced by the identity $\ker(d_{k,\max,\infty})\rightarrow \ker(d_{k,\max,\infty})$. We recall now some well known properties that will be frequently used later:\\

\noindent \textbf{a)} Let $1<p<\infty$ and let $p'=p/(p-1)$. Let $\omega\in L^p\Omega^k(M,g)$. Then $\omega\in \mathcal{D}(d_{k,\max,p})$ and $L^{p}\Omega^{k+1}(M,g)\ni\psi=d_{k,\max,p}\omega$ if and only if for any $\eta\in \mathcal{D}(d_{m-k-1,\min,p'})$ we have 
\begin{equation}
\label{cope}
\int_M\psi\wedge \eta=\int_{M}\omega\wedge d_{m-k-1,\min,p'}\eta.
\end{equation}
Analogously, given an arbitrary  $\omega\in L^p\Omega^k(M,g)$, we have $\omega\in \mathcal{D}(d_{k,\min,p})$ and $L^{p}\Omega^{k+1}(M,g)\ni\psi=d_{k,\min,p}\omega$ if and only if for any $\eta\in \mathcal{D}(d_{m-k-1,\max,p'})$ we have 
\begin{equation}
\label{naghen}
\int_M\psi\wedge \eta=\int_{M}\omega\wedge d_{m-k-1,\max,p'}\eta.
\end{equation}

 \noindent \textbf{b)} Let $1<p<\infty$ and let $p'=p/(p-1)$. Then the bilinear map
\begin{equation}
\label{stock}
\overline{H}^k_{p,\max}(M,g)\times \overline{H}^{m-k}_{p',\min}(M,g)\rightarrow \mathbb{R},\ ([\omega],[\eta])\mapsto \int_M\omega\wedge\eta
\end{equation}
is a well-defined and non-degenerate pairing.\\

\noindent \textbf{c)} Let $1\leq p\leq q \leq \infty$. Assume that $\vol_g(M)<\infty$. It is well known that if $\omega\in L^q\Omega^k(M,g)$ then $\omega\in L^p\Omega^k(M,g)$ and that the corresponding inclusion $i:L^q\Omega^k(M,g)\hookrightarrow L^p\Omega^k(M,g)$ is continuous. Then it is immediate to check that if $\omega\in \mathcal{D}(d_{k,\max,q})$ then $\omega=i(\omega)\in \mathcal{D}(d_{k,\max,p})$ and 
\begin{equation}
\label{corno}
i\circ d_{k,\max,q}=d_{k,\max,p}\circ i\ \text{on}\ \mathcal{D}(d_{k,\max,q}).
\end{equation}

Similarly if  $1\leq p\leq q < \infty$ and $\omega\in \mathcal{D}(d_{k,\min,q})$ then $\omega=i(\omega)\in \mathcal{D}(d_{k,\min,p})$ and 
\begin{equation}
\label{delcatria}
i\circ d_{k,\min,q}=d_{k,\min,p}\circ i\ \text{on}\ \mathcal{D}(d_{k,\min,q}).
\end{equation}

\noindent \textbf{d)} Let $1\leq p\leq q \leq \infty$ and let $U\subset M$ be an open subset. Let $\omega\in \mathcal{D}(d_{k,\max,p})\subset L^{p}\Omega^k(M,g)$. Then 
\begin{equation}
\label{otoedoda}
\omega|_U\in \mathcal{D}(d_{k,\max,p})\subset L^{p}\Omega^k(U,g|_U)\ \text{and}\ 
d_{k,\max,p}(\omega|_U)=d_{k,\max,p}\omega|_U.
\end{equation}

We continue now by giving the following definitions.

\begin{defi}
Let $(M,g)$ be a possibly incomplete Riemannian manifold. Let $1\le p<\infty$. We will say that the $L^p$-Stokes theorem holds on $L^p\Omega^k(M,g)$ if the following two operators $$d_{k,\max,p}:L^p\Omega^k(M,g)\rightarrow L^p\Omega^{k+1}(M,g)\ \text{and}\ d_{k,\min,p}:L^p\Omega^k(M,g)\rightarrow L^p\Omega^{k+1}(M,g)$$ coincide.
\end{defi}
A well known result says that when $(M,g)$ is complete then the $L^p$-Stokes theorem holds true for any $k=0,...,m$. We will recall later a proof of this result, see Prop. \ref{Lpmaxmin}. Conversely there are examples of incomplete Riemannian manifolds where the $L^p$-Stokes theorem fails to be true, see for instance \cite{BL}, \cite{BLE}, \cite{JCH} and \cite{GL}.

\begin{defi}
Let $(M,g)$ be a possibly incomplete Riemannian manifold. Let $1\le p<\infty$. We will say that the $L^p$-divergence theorem holds on $(M,g)$ if the following two operators $$d_{0,\max,p}^t:L^p\Omega^1(M,g)\rightarrow L^p(M,g)\ \text{and}\ d_{0,\min,p}^t:L^p\Omega^1(M,g)\rightarrow L^p(M,g)$$ coincide.
\end{defi}
In the above definition $d_{0,\max,p}^t:L^p\Omega^1(M,g)\rightarrow L^p(M,g)$ and $d_{0,\min,p}^t:L^p\Omega^1(M,g)\rightarrow L^p(M,g)$ are the $L^p$-maximal/minimal extensions of $d^t_0:\Omega^1_c(M)\rightarrow C^{\infty}_c(M)$. Analogously  to \eqref{ner} and \eqref{spinaci} $d_{0,\max,p}^t:L^p\Omega^1(M,g)\rightarrow L^p(M,g)$ is defined in the distributional sense while $d_{0,\min,p}^t:L^p\Omega^1(M,g)\rightarrow L^p(M,g)$ is the graph closure of $d_0^t:\Omega^1_c(M,g)\rightarrow C^{\infty}_c(M)$.

\begin{defi}
\label{parax}
Let $(M,g)$ be a possibly incomplete Riemannian manifold.  Let $q\in [1,\infty)$. Then $(M,g)$ is said to be $q$-parabolic if there exists a sequence of Lipschitz functions with compact support $\{\phi_n\}_{n\in \mathbb{N}}\subset \mathrm{Lip}_c(M,g)$ such that
\begin{enumerate}
\item $0\leq \phi_n\leq 1$;
\item $\phi_n\rightarrow 1$ almost everywhere as $n\rightarrow \infty$;
\item $\|d_0\phi_n\|_{L^q\Omega^1(M,g)}\rightarrow 0$ as $n\rightarrow \infty$.
\end{enumerate}
\end{defi}

\begin{defi}
Let $(M,g)$ be a possibly incomplete Riemannian manifold.  Then $(M,g)$ is said to be stochastically complete if $$e^{-t\Delta_0^{\mathcal{F}}}1=1$$ where $\Delta_0^{\mathcal{F}}:L^2(M,g)\rightarrow L^2(M,g)$ is the Friedrich extension of the Laplace-Beltrami operator $\Delta_0:C^{\infty}_c(M)\rightarrow C^{\infty}_c(M)$ and $e^{-t\Delta_0^{\mathcal{F}}}:L^2(M,g)\rightarrow L^2(M,g)$ is the heat operator associated to $\Delta_0^{\mathcal{F}}$.
\end{defi}
We add a small remark to the above definition. The heat operator $e^{-t\Delta_0^{\mathcal{F}}}$
is an integral operator with a positive smooth kernel $k(t,x,y)$, $(t,x,y)\in (0,\infty)\times M\times M$, such that $\int_Mk(t,x,y)\dvol_g(y)\leq 1$ for all $t>0$ and $x\in M$. Hence, given a bounded function $f\in L^{\infty}(M)$, the function $e^{-t\Delta_0^{\mathcal{F}}}f(x):=\int_Mk(t,x,y)f(y)\dvol_g(y)$ is well defined and it is still  bounded.\\ We invite the reader to consult \cite{GYA} and \cite{paraTroya} for an in-depth treatment about $q$-parabolicity and stochastic completeness. The definition of the heat operator can be found for instance in \cite{BGV} or \cite{GYA}.\\   

We conclude this section by recalling few basic notions concerning {\em almost K\"ahler manifolds}, see e.g. \cite{ToWeYa} for more details. Let $(M,\omega)$ be a symplectic manifold of dimension $2m$. It is well known that there exists a compatible almost-complex structure $J\in \mathrm{End}(TM)$, that is an endomorphism of tangent bundle of $M$ satisfying $J^2=-\text{Id}$ and such that  $g(X,Y):=\omega(X,JY)$  is a Riemannian metric on $M$, where  $X,Y\in \frak{X}(M)$, see e.g. \cite{Cannas}. Moreover the symplectic form $\omega$ and the volume form $\dvol_g$ are related by the following formula: $m!\dvol_g=\omega^m$.

\begin{defi}
\label{almost}
An almost K\"ahler manifold $(M,\omega,J,g)$ is given by a symplectic manifold $(M,\omega)$ endowed with a compatible almost-complex structure $J$ and the Riemannian metric $g$ defined as $g(X,Y):=\omega(X,JY)$ for any  $X,Y\in \frak{X}(M)$.
\end{defi}

From now on by saying that an almost K\"ahler manifold $(M,\omega,J,g)$ is complete we will mean that $g$ is a complete metric on $M$.

\begin{prop}
\label{limitato}
Let $(M,\omega,J,g)$ be an almost K\"ahler manifold. Then $\omega\in L^{\infty}\Omega^2(M,g)$
\end{prop}

\begin{proof}
Let $\nabla:C^{\infty}(M,TM)\rightarrow C^{\infty}(M,T^*M\otimes TM)$ be the {\em second canonical connection}. We refer to \cite{Gaud}, \cite{ToWeYa} and the reference therein for more details. Here it is enough to recall that $\nabla$ is a connection compatible with both $g$ and $J$, that is $\nabla g=0$ and $\nabla J=0$.  It is easy to check now that these two properties entail that $\omega$ is parallel with respect to $\nabla$. Finally, as $\nabla$ is compatible with $g$, we have
$d_0(g(\omega,\omega))=0$, and thus we can conclude that $\omega\in L^{\infty}\Omega^2(M,g)$ as desired.
\end{proof}

\section{Main results}

This section is split in two subsections. In the first one we collect various technical results that will be needed later. The second one is devoted to the main results of this paper.

\subsection{Some technical propositions}
We have the following propositions:

\begin{prop}
\label{littletool}
Let $(N,g)$ be an oriented and  possibly incomplete Riemannian manifold of dimension $n$.
\begin{enumerate}
\item Let $\eta\in L^{1}\Omega^n(N,g)$. Then $\int_N\eta<\infty$.
\item Let $\eta\in L^{1}\Omega^n(N,g)$ such that $\eta\in \overline{\im(d_{n-1,\min,1})}$. Then $\int_N\eta=0$.
\end{enumerate}
\end{prop}
\begin{proof}
As usual let $\dvol_g$ be the volume form of $g$. Then there exists a function $f\in L^1(N,g)$ such that $f\dvol_g=\eta$. We have $$|\int_N\eta|= |\int_Nf\dvol_g|\leq\int_N|f|\dvol_g=\|\eta\|_{L^1\Omega^n(N,g)}<\infty.$$ This establishes the first point. For the second point, as $\eta\in \overline{\im(d_{n-1,\min,1})}$, we know that there exists a sequence of smooth forms with compact support  $\{\beta_k\}_{k\in \mathbb{N}}\subset \Omega_c^{n-1}(N)$ such that $\lim d_{n-1}\beta_k=\eta$ in $L^{1}\Omega^n(N,g)$ as $k\rightarrow \infty$. Moreover we recall that $L^{\infty}(N)$ is the dual Banach space of $L^1\Omega^n(N,g)$. Therefore we have $\int_N\eta=\lim\int_Nd_{n-1}\beta_k$ as $k\rightarrow \infty$. But $\int_Nd_{n-1}\beta_k=0$ for each $k\in \mathbb{N}$ thanks to the Stokes theorem. Hence $\int_N\eta=0$ as well. 
\end{proof}

\begin{prop}
\label{olimpia}
Let $(N,g)$ be a possibly incomplete Riemannian manifold of dimension $n$. Let $\omega\in \mathcal{D}(d_{k,\max,p})\subset L^p\Omega^k(N,g)$ for some $k=0,1,...,n$ and $1\leq p<\infty$. Assume that there exists an open subset $U\subset N$ with compact closure such that $\omega|_{N\setminus \overline{U}}\equiv0$. Then $\omega\in \mathcal{D}(d_{k,\min,p})$. 
\end{prop}
\begin{proof}
According to \cite{GoTr} Th. 12.5 we know that there exists a sequence $\{\omega_j\}\subset \Omega^k(N)\cap L^p\Omega^k(N,g)$ such that $\omega_j\rightarrow \omega$ as $j\rightarrow \infty$ in the graph norm of $\mathcal{D}(d_{k,\max,p})$. Let $V$ be an open neighborhood of $\overline{U}$ with $\overline{V}$ compact. Let $\chi\in C^{\infty}_c(N)$, $\chi:N\rightarrow [0,1]$, be a smooth function with compact support such that  $\chi(x)=1$ for any $x\in V$. Then, using the fact that $\chi\omega=\omega$, $\chi d_{k,\max,p}\omega=d_{k,\max,p}\omega$ and $d_0\chi\wedge \omega\equiv0$, it is an easy exercise to check that the sequence $\{\chi\omega_j\}$ converges to $\omega$  in the graph norm of $\mathcal{D}(d_{k,\max,p})$.  Finally, as $\{\chi\omega_j\}\subset \Omega^k_c(N)$, we can conclude that $\omega\in \mathcal{D}(d_{k,\min,p})$ as desired.
\end{proof}

\begin{prop}
\label{tecprop}
Let $(N,g)$ be a possibly incomplete Riemannian manifold of dimension $n$.  Let $\gamma\in \mathcal{D}(d_{r,\max,\infty})\subset  L^{\infty}\Omega^r(N,g)$.  Then the following properties hold true for any $k=0,1,...,n$ and $p\in [1,\infty)$:
\begin{itemize}
\item If $\omega\in \mathcal{D}(d_{k,\max,p})$ then $\gamma\wedge\omega\in \mathcal{D}(d_{k+r,\max,p})$ and $d_{k+r,\max,p}(\gamma\wedge\omega)=d_{r,\max,\infty}\gamma\wedge\omega+(-1)^r\gamma \wedge d_{k,\max,p}\omega$. 
\item If f $\omega\in \mathcal{D}(d_{k,\min,p})$ then $\gamma\wedge\omega\in \mathcal{D}(d_{k+r,\min,p})$ and $d_{k+r,\min,p}(\gamma\wedge\omega)=d_{r,\max,\infty}\gamma\wedge\omega+(-1)^r\gamma \wedge d_{k,\min,p}\omega$. 
\end{itemize}
\end{prop}

\begin{proof}
Let $\omega\in \mathcal{D}(d_{k,\max,p})$. We first point out that $d_{r,\max,\infty}\gamma\wedge\omega+(-1)^r\gamma \wedge d_{k,\max,p}\omega\in L^p\Omega^{r+k+1}(N,g)$. This follows easily by the fact that $\gamma\in L^{\infty}\Omega^r(N,g)$, $d_{r,\max,\infty}\gamma\in L^{\infty}\Omega^{r+1}(N,g)$, $\omega\in L^p\Omega^{k}(N,g)$ and $d_{k,\max,p}\omega\in L^p\Omega^{k+1}(N,g)$. According to \cite{GoTr} Th. 12.5  we know that there is a sequence $\{\omega_j\}\subset \Omega^{k}(N)\cap \mathcal{D}(d_{k,\max,p})$ such that $\omega_j\rightarrow \omega$ as $j\rightarrow \infty$ in the graph norm of $d_{k,\max,p}$. Assume for the moment that 
\begin{equation}
\label{riccinon}
\{\gamma\wedge\omega_j\}\subset \mathcal{D}(d_{k+r,\max,p})\ \text{and that}\ d_{k+r,\max,p}(\gamma\wedge\omega_j)=d_{r,\max,\infty}\gamma\wedge\omega_j+(-1)^r\gamma \wedge d_k\omega_j.
\end{equation}
It is immediate to check that $\gamma\wedge\omega_j\rightarrow \gamma\wedge\omega$  in $L^p\Omega^{r+k}(N,g)$ and $d_{r,\max,\infty}\gamma\wedge\omega_j+(-1)^r\gamma \wedge d_k\omega_j\rightarrow d_{r,\max,\infty}\gamma\wedge\omega+(-1)^r\gamma \wedge d_{k,\max,p}\omega$ in $L^p\Omega^{r+k+1}(N,g)$ as $j\rightarrow \infty$. Thus we can deduce that $$\gamma\wedge\omega\in \mathcal{D}(d_{r+k,\max,p})\ \text{and}\ d_{r+k,\max,p}(\gamma\wedge \omega)=d_{r,\max,\infty}\gamma\wedge\omega+(-1)^r\gamma \wedge d_{k,\max,p}\omega.$$ In order to complete the proof of the first point we are left  to show \eqref{riccinon}. To this aim it is enough to show that for each point $x\in N$ there exists an open  neighborhood $U$ such that $(\gamma\wedge\omega_j)|_U\in \mathcal{D}(d_{k+r,\max,p})\subset L^p\Omega^{k+r}(U,g|_U)$ and 
$(d_{r+k,\max,p}(\gamma\wedge\omega_j))|_U=(d_{r,\max,\infty}\gamma\wedge\omega_j+(-1)^r\gamma \wedge d_k\omega_j)|_{U}$. Therefore let $U$ be an open subset of $N$ with compact closure. Clearly for any $p\in [1,\infty)$ $\gamma|_U\in \mathcal{D}(d_{r,\max,p})\subset L^p\Omega^r(U,g|_U)$ and $d_{r,\max,p}(\gamma|_U)=d_{r,\max,\infty}(\gamma|_U)=(d_{r,\max,\infty}\gamma)|_U$, see \eqref{corno} and \eqref{otoedoda}. Using again \cite{GoTr} Th. 12.5 we know that there is a sequence $\{\gamma_n\}\subset \Omega^r(U)\cap \mathcal{D}(d_{r,\max,p})\subset L^p\Omega^r(U,g|_U)$ such that $\gamma_n\rightarrow \gamma|_U$ in $L^p\Omega^r(U,g|_U)$ and $d_r\gamma_n\rightarrow (d_{r,\max,\infty}\gamma)|_U$ in $L^p\Omega^{r+1}(U,g|_U)$ as $n\rightarrow \infty$. Consider now the sequence $\{\gamma_n\wedge\omega_j|_U\}_{n\in \mathbb{N}}$. Since $\overline{U}$ is compact and $\omega_j\in \Omega^k(M)$ we have $\omega_j|_U\in L^{\infty}\Omega^k(U,g|_U)$ and $d_k\omega_j|_U\in L^{\infty}\Omega^{k+1}(U,g|_U)$; this allows us to get immediately that  $\gamma_n\wedge\omega_j|_U\in L^p\Omega^{r+k}(U,g|_U)$, $d_{r+k}(\gamma_n\wedge \omega_j|_U)\in L^p\Omega^{r+k+1}(U,g|_U)$, $\gamma_n\wedge\omega_j|_U\rightarrow (\gamma\wedge\omega_j)|_U$ in $L^p\Omega^{r+k}(U,g|_U)$ and $d_{r+k}(\gamma_n\wedge\omega_j|_U)\rightarrow (d_{r,\max,\infty}\gamma\wedge\omega_j+(-1)^r\gamma \wedge d_k\omega_j)|_U$ in $L^p\Omega^{k+1}(U,g|_U)$ as $n\rightarrow \infty$. Therefore $$(\gamma\wedge\omega_j)|_U\in \mathcal{D}(d_{r+k,\max,p})\subset L^p\Omega^{r+k}(U,g|_U)\ \text{and}\ d_{r+k,\max,p}(\gamma\wedge\omega_j)|_U= (d_{r,\max,\infty}\gamma\wedge\omega_j+(-1)^r\gamma \wedge d_k\omega_j)|_U.$$ In other words the restriction over $U$ of $d_{k+r}(\gamma\wedge\omega_j)$, where  $d_{k+r}(\gamma\wedge\omega_j)$ is understood in the distributional sense, is given by $(d_{r,\max,\infty}\gamma\wedge\omega_j+(-1)^r\gamma \wedge d_k\omega_j)|_U$. Thus the proof of the first statement is complete. The second statement follows by arguing as for the first one. In particular, as $\omega\in \mathcal{D}(d_{k,\min,p})$, we can take  $\{\omega_j\}\subset \Omega^{k}_c(M)$ such that $\omega_j\rightarrow \omega$ as $j\rightarrow \infty$ in the graph norm of $d_{k,\min,p}$. By the first point we know that $\gamma\wedge \omega\in \mathcal{D}(d_{k,\max,p})$, $\{\gamma\wedge\omega_j\}\subset \mathcal{D}(d_{r+k,\max,p})$ and $\gamma\wedge\omega_j\rightarrow \gamma\wedge\omega$ in $\mathcal{D}(d_{r+k,\max,p})$ with respect to the corresponding graph norm. Finally, since $\gamma\wedge\omega_j$ has compact support we can use Prop. \ref{olimpia} to conclude that $\{\gamma\wedge\omega_j\}\subset \mathcal{D}(d_{k,\min,p})$ and this completes the proof the second statement. 
\end{proof}

\begin{prop}
\label{coredomain}
Let $(M,g)$ be a possibly incomplete Riemannian manifold. Then $W^{1,p}(M,g)\cap L^{\infty}(M)\cap C^{\infty}(M)$ is dense in $W^{1,p}(M,g)$ for any $1\leq p<\infty$.
\end{prop}
\begin{proof}
We adapt to our case the prof of Prop. 2.5 in \cite{FrB}. According to \cite{GoTr} Th. 12.5 we know that $W^{1,p}(M,g)\cap C^{\infty}(M)$ is dense in $W^{1,p}(M,g)$ for any $1\leq p<\infty$. Therefore it is enough to show that $W^{1,p}(M,g)\cap L^{\infty}(M)\cap C^{\infty}(M)$ is dense in $W^{1,p}(M,g)\cap C^{\infty}(M)$ with respect to the norm of $W^{1,p}(M,g)$. Let $f\in W^{1,p}(M,g)\cap C^{\infty}(M)$ and let $$f_n:=\frac{f}{\left(\frac{f^2}{n}+1\right)^{\frac{1}{2}}}.$$ Clearly we have $f_n\in C^{\infty}(M)$ and $(\frac{f^2}{n}+1)^{-\frac{1}{2}}\in L^{\infty}(M)$. Thus we get $f_n\in L^p(M,g)$. Moreover we have $$|f_n|=\frac{|f|}{\left(\frac{f^2}{n}+1\right)^{\frac{1}{2}}}\leq \sqrt{n}.$$ We can thus conclude that $f_n\in L^{\infty}(M)$ for each $n\in \mathbb{N}$. Concerning $df_n$ we have $$df_n=\left(\frac{f^2}{n}+1\right)^{-\frac{1}{2}}df-\frac{1}{n}\left(\frac{f^2}{n}+1\right)^{-3/2}f^2df.$$ As $(\frac{f^2}{n}+1)^{-\frac{1}{2}}\in L^{\infty}(M)$ we get $(\frac{f^2}{n}+1)^{-\frac{1}{2}}df\in L^p\Omega^1(M,g)$. For the other term, $\frac{1}{n}(\frac{f^2}{n}+1)^{-3/2}f^2df$, we have  $n/(f^2+n)\leq 1$ which in turn implies $(1+f^2/n)^{-3/2}\leq n/(f^2+n)$. In this way we get 
\begin{equation}
\label{barrabas}
\frac{1}{n}(\frac{f^2}{n}+1)^{-3/2}f^2\leq \frac{f^2}{f^2+n}\leq 1.
\end{equation}
 Since $df\in L^p\Omega^1(M,g)$ we can thus deduce that $\frac{1}{n}(\frac{f^2}{n}+1)^{-3/2}f^2df\in L^p\Omega^1(M,g)$. Therefore $df_n\in L^p\Omega^1(M,g)$. Summarizing we showed that $\{f_n\}_{n\in \mathbb{N}}\subset W^{1,p}(M,g)\cap L^{\infty}(M)\cap C^{\infty}(M)$. Finally we are left to prove that $\{f_n\}$ converges to $f$ in $W^{1,p}(M,g)$. Concerning $\|f_n-f\|_{L^{p}(M,g)}^p$ we have $$\|f_n-f\|_{L^{p}(M,g)}^p=\int_M\left|1-(\frac{f^2}{n}+1)^{-\frac{1}{2}}\right|^p|f|^p\dvol_g.$$ As $1-(\frac{f^2}{n}+1)^{-\frac{1}{2}}\in L^{\infty}(M)$ we can use the Lebesgue dominate convergence theorem to conclude that $f_n\rightarrow f$ in $L^p(M,g)$ as $n\rightarrow \infty$. We come now to $\|df_n-df\|_{L^{p}\Omega^1(M,g)}^p$. We have $$\|df_n-df\|_{L^p\Omega^1(M,g)}^p\leq \|df-\left(\frac{f^2}{n}+1\right)^{-\frac{1}{2}}df\|_{L^p\Omega^1(M,g)}+\|\frac{1}{n}\left(\frac{f^2}{n}+1\right)^{-3/2}f^2df\|_{L^p\Omega^1(M,g)}.$$ Again by virtue of the Lebesgue dominate convergence theorem we have $$\lim_{n\rightarrow \infty}\|df-\left(\frac{f^2}{n}+1\right)^{-\frac{1}{2}}df\|_{L^p\Omega^1(M,g)}=0.$$ Finally for the remaining term we have $$\|\frac{1}{n}\left(\frac{f^2}{n}+1\right)^{-3/2}f^2df\|_{L^p\Omega^1(M,g)}=\int_M\left|\frac{1}{n}\left(\frac{f^2}{n}+1\right)^{-3/2}f^2\right|^p|df|_g^p.$$ Thanks to \eqref{barrabas} we can use again the Lebesgue dominate convergence theorem to deduce that $$\lim_{n\rightarrow \infty}\|\frac{1}{n}\left(\frac{f^2}{n}+1\right)^{-3/2}f^2df\|_{L^p\Omega^1(M,g)}=0.$$ In conclusion we proved that $f_n\rightarrow f$ in $W^{1,p}(M,g)$ as $n\rightarrow \infty$. The proof is thus complete.
\end{proof}

\begin{prop}
\label{messico}
Let $(M,g)$ be a possibly incomplete Riemannian manifold of dimension $m$. Assume that $(M,g)$ is $p$-parabolic with $p\in [1,\infty)$. Then $L^{\infty}\Omega^k(M,g)\cap \mathcal{D}(d_{k,\max,p})\subset \mathcal{D}(d_{k,\min,p})$ for any $k=0,...,m$.
\end{prop}

\begin{proof}
Let $\eta\in L^{\infty}\Omega^k(M,g)\cap \mathcal{D}(d_{k,\max,p})$ and let $\{\phi_n\}$ be a sequence of functions that makes $(M,g)$ $p$-parabolic. Consider the sequence $\{\phi_n\eta\}$. Thanks to Prop. \ref{tecprop} we know that $\{\phi_n\eta\}\subset \mathcal{D}(d_{k,p,\max})$ and $d_{k,p,\max}(\phi_n\eta)=d_0\phi_n\wedge \eta+\phi_nd_{k,p,\max}\eta$.
As $0\leq \phi_n\leq 1$ for each $n\in \mathbb{N}$ we can use the Lebesgue dominate convergence theorem to conclude that $\phi_n\eta\rightarrow \eta$ and $\phi_nd_{k,\max,p}\eta\rightarrow d_{k,\max,p}$ in $L^p\Omega^k(M,g)$ and $L^p\Omega^{k+1}(M,g)$ as $n\rightarrow \infty$. By using the fact that $\|d_0\phi_n\|_{L^p\Omega^1(M,g)}\rightarrow 0$ as $n\rightarrow \infty$ and that $\eta\in L^{\infty}\Omega^k(M,g)\cap \mathcal{D}(d_{k,\max,p})$ we can conclude that $$\lim_{n\rightarrow \infty}\|d_0\phi_n\wedge \eta\|_{L^p\Omega^{k+1}(M,g)}\leq \lim_{n\rightarrow \infty}\|d_0\phi_n\|_{L^p\Omega^1(M,g)}\| \eta\|_{L^{\infty}\Omega^{k}(M,g)}=0.$$ Thus we showed that  $\phi_n\eta\rightarrow \eta$ in  $\mathcal{D}(d_{k,\max,p})$ endowed with the corresponding graph norm as $n\rightarrow \infty$. Finally, thanks to Prop. \ref{olimpia}, we know that $\{\phi_n\eta\}\subset \mathcal{D}(d_{k,\min,p})$ and this allows us to conclude that $\eta \in \mathcal{D}(d_{k,\min,p})$ too. The proof is thus complete.
\end{proof}

\begin{prop}
\label{paraSob}
Let $(M,g)$ be a possibly incomplete Riemannian manifold and let $p\in [1,\infty)$. 
\begin{enumerate} 
\item If $(M,g)$ is $p$-parabolic then $W^{1,p}(M,g)=W^{1,p}_0(M,g)$.
 \item If $\vol_g(M)<\infty$ then $(M,g)$ is $p$-parabolic if and only if $W^{1,p}(M,g)=W^{1,p}_0(M,g)$.
\end{enumerate}
\end{prop}
\begin{proof}
These properties are well known. For the sake of the reader we give a proof. According to Prop. \ref{coredomain} it is enough to show that $W^{1,p}(M,g)\cap L^{\infty}(M)\cap C^{\infty}(M)\subset W^{1,p}_0(M,g)$. Let $f\in W^{1,p}(M,g)\cap L^{\infty}(M)\cap C^{\infty}(M)$ and let $\{\phi_n\}$ be a sequence of Lipschitz functions with compact support that makes $(M,g)$ $p$-parabolic. We want to show  that $\{f\phi_n\}\subset W^{1,p}_0(M,g)$ and $f\phi_n\rightarrow f$ in $W^{1,p}(M,g)$ as $n\rightarrow \infty$ . The inclusion  $\{f\phi_n\}\subset W^{1,p}_0(M,g)$ follows immediately by Prop. \ref{olimpia} and Prop. \ref{tecprop} while the convergence $f\phi_n\rightarrow f$ in $W^{1,p}(M,g)$ is a direct consequence of Prop. \ref{messico}.
Thus we proved that $f\in W_0^{1,p}(M,g)$ and so we can conclude that $W^{1,p}(M,g)=W^{1,p}_0(M,g)$. This shows the first point.  Assume now that  $W^{1,p}(M,g)=W^{1,p}_0(M,g)$ and that  $(M,g)$ has finite volume. Then there exists a sequence $\{\phi_n\}\subset C^{\infty}_c(M)$ such that $\phi_n\rightarrow 1$ in $W^{1,p}(M,g)$ as $n\rightarrow \infty$. As it is well known we can pick up a subsequence $\{\tilde{\phi}_n\}\subset \{\phi_n\}$ such that  $\tilde{\phi}_n\rightarrow 1$ pointwise almost everywhere as $n\rightarrow \infty$. Let us  define $\psi_n:=\min\{\tilde{\phi}_n,1\}$. Then it is straightforward to check that $\{\psi_n\}$ is a sequence of Lipschitz functions with compact support that makes $(M,g)$ $p$-parabolic.
\end{proof}

\subsection{Non-vanishing theorems}
This subsection is concerned with various non-vanishing theorems for the (reduced) $L^p$-cohomology of a possibly incomplete almost K\"ahler manifold of finite volume.  We collect various criteria inspired by the notion of $q$-parabolicity and related properties. We start with the following: 

\begin{teo}
\label{calamari}
Let $(M,\omega,J,g)$ be a possibly incomplete almost K\"ahler manifold of finite volume and dimension $2m>0$.  Assume that $\overline{H}^{2m}_{p,\max}(M,g)=\overline{H}^{2m}_{p,\min}(M,g)$ for some $1\leq p<\infty$. Then $\omega^k$ induces a non trivial class in $\overline{H}^{2k}_{q,\max}(M,g)$ for any $k=1,2,...,m$ and $q\in [p,\infty]$.
\end{teo}

\begin{proof}
First we start with few remarks. Thanks to Prop. \ref{limitato} we know that $\omega\in L^{\infty}\Omega^2(M,g)$. Clearly for each $k=1,2,...,m$, $\omega^k\in \Omega^{2k}(M)\cap L^{\infty}\Omega^{2k}(M,g)$. Since $\vol_h(M)<\infty$ we have $\omega^k\in \Omega^{2k}(M)\cap L^s\Omega^{2k}(M,g)$ for any $s\in [1,\infty]$. Therefore $\omega^k\in \ker(d_{2k,\max,s})$ for any $s\in [1,\infty]$, that is $\omega^k$ induces a class in  $\overline{H}^{2k}_{s,\max}(M,g)$ for any $s\in [1,\infty]$.
Furthermore, as $\omega^k\in L^{\infty}\Omega^{2k}(M,g)$, we have that the map $L^k:L^p\Omega^r(M,g)\rightarrow L^p\Omega^{r+2k}(M,g)$ given by $\eta\mapsto \omega^k\wedge \eta$ is bounded.  Now we tackle the proof of the statement.  By contrast let us assume that $[\omega^k]=0$ in $\overline{H}^{2k}_{q,\max}(M,g)$ for some $k\in\{1,2,...,m\}$ and some $q\in [p,\infty]$. This means that there exists a sequence $\{\phi_j\}\subset \mathcal{D}(d_{2k-1,\max,q})\subset L^q\Omega^{2k-1}(M,g)$ such that $$\lim_{j\rightarrow \infty} d_{2k-1,\max,q}\phi_j=\omega^k$$ in $L^q\Omega^{2k}(M,g)$. Let us now consider the sequence $\{\phi_j\wedge \omega^{m-k}\}$. According to Prop. \ref{tecprop} we have $\{\phi_j\wedge \omega^{m-k}\}\subset \mathcal{D}(d_{2m-1,\max,q})$ and $d_{2m-1,\max,q}(\phi_j\wedge \omega^{m-k})=(d_{2k-1,\max,q}\phi_j)\wedge \omega^{m-k}$. Hence $$\lim_{j\rightarrow \infty} d_{2m-1,\max,q}(\phi_j\wedge \omega^{m-k})=\omega^m$$ in $L^q\Omega^{2m}(M,g)$. By the fact that $\overline{H}^{2m}_{p,\max}(M,g)=\overline{H}^{2m}_{p,\min}(M,g)$, $\overline{H}^{2m}_{p,\max}(M,g)=L^p\Omega^{2m}(M,g)/\overline{\im(d_{2m-1,p,\max})}$\\ and $\overline{H}^{2m}_{p,\min}(M,g)=L^p\Omega^{2m}(M,g)/\overline{\im(d_{2m-1,p,\min})}$, we can deduce that  $$\overline{\im(d_{2m-1,\max,p})}= \overline{\im(d_{2m-1,\min,p})}.$$  Since $(M,g)$ has finite volume and $q\geq p$ we have in particular that  $\{\phi_j\wedge \omega^{m-k}\}\subset \mathcal{D}(d_{2m-1,\max,p})$ and  $$\lim_{j\rightarrow \infty} d_{2m-1,\max,p}(\phi_j\wedge \omega^{m-k})=\omega^m$$ in $L^p\Omega^{2m}(M,g)$, see \eqref{corno}. Hence we get that $\omega^m\in \overline{\im(d_{2m-1,\max,p})}\subset L^p\Omega^{2m}(M,g)$. This, together with the fact that $\overline{\im(d_{2m-1,\max,p})}=\overline{\im(d_{2m-1,\min,p})}$, implies the existence of a sequence $\{\psi_j\}\subset \Omega_c^{2m-1}(M)$ such that $\lim d_{2m-1}\psi_j=\omega^m$ in $L^p\Omega^{2m}(M,g)$ as $j\rightarrow \infty$. Using again that $(M,g)$ has finite volume we can conclude that $$\lim_{j\rightarrow \infty} d_{2m-1}\psi_j=\omega^m$$ in $L^1\Omega^{2m}(M,g)$. Finally, according to Prop. \ref{littletool}, this implies that $$m!\vol_g(M)=\int_M\omega^m=0$$ which is clearly absurd. We can thus conclude that $[\omega^k]\neq 0$ in $\overline{H}^{2k}_{q,\max}(M,g)$ for each  $k\in\{1,2,...,m\}$ and for any $q\in [p,\infty]$.  
\end{proof}

In the following propositions we collect many properties that imply the assumptions of Th. \ref{calamari}

\begin{prop}
\label{LpStokes}
Let $(M,\omega,J,g)$ be a possibly incomplete almost K\"ahler manifold of finite volume and dimension $2m>0$.  Assume that the $L^p$-Stokes theorem holds on $L^p\Omega^{2m-1}(M,g)$ for some $1\leq p<\infty$. Then $\omega^k$ induces a non trivial class in $\overline{H}^{2k}_{q,\max}(M,g)$ for any $k=1,2,...,m$ and for any $q\in [p,\infty]$.
\end{prop}

\begin{proof}
 As we assumed that the $L^p$-Stokes theorem holds at the level of $(2m-1)$-forms we know that the following two operators coincide $$d_{2m-1,\max,p}:L^p\Omega^{2m-1}(M,g)\rightarrow L^p\Omega^{2m}(M,g)\ \mathrm{and}\ d_{2m-1,\min,p}:L^p\Omega^{2m-1}(M,g)\rightarrow L^p\Omega^{2m}(M,g).$$ Hence we know in particular that $\overline{\im(d_{2m-1,\max,p})}= \overline{\im(d_{2m-1,\min,p})}$ which in turn implies immediately $$\overline{H}^{2m}_{p,\max}(M,g)=\overline{H}^{2m}_{p,\min}(M,g).$$ As $(M,g)$ has finite volume and $q\geq p$ the conclusion now follows immediately by Th. \ref{calamari}.
\end{proof}

\begin{prop}
\label{Lpdiver}
Let $(M,\omega,J,g)$ be a possibly incomplete almost K\"ahler manifold of finite volume and dimension $2m>0$. Assume that the $L^p$-divergence theorem holds for $(M,g)$ for some $1\leq p<\infty$. Then $\omega^k$ induces a non trivial class in $\overline{H}^{2k}_{q,\max}(M,g)$ for any $k=1,2,...,m$ and for any $q\in [p,\infty]$.
\end{prop}

\begin{proof}
Let $*:\Omega^k_c(M)\rightarrow \Omega^{2m-k}_c(M)$ be the Hodge star operator induced by $g$. It is easy to check that $*:\Omega^k_c(M)\rightarrow \Omega^{2m-k}_c(M)$ extends as a bounded and bijective operator $*:L^p\Omega^k(M,g)\rightarrow L^{p}\Omega^{2m-k}(M,g)$ such that $\|\eta\|_{L^p\Omega^k(M,g)}=\|*\eta\|_{L^p\Omega^{2m-k}(M,g)}$ for any $\eta\in L^{p}\Omega^k(M,g)$. Using  $*:L^p\Omega^k(M,g)\rightarrow L^{p}\Omega^{2m-k}(M,g)$ it is easy to show that for each $k=0,1,...,2m$ we have $$\pm*d_{k,\max/\min,p}*=d^t_{2m-k-1,\max/\min,p}$$ where the sign depends on the parity of $k$. Thus the fact that the $L^p$-divergence theorem holds entails that the two operators $$d_{2m-1,\max,p}:L^p\Omega^{2m-1}(M,g)\rightarrow L^p\Omega^{2m}(M,g)\ \mathrm{and}\ d_{2m-1,\min,p}:L^p\Omega^{2m-1}(M,g)\rightarrow L^p\Omega^{2m}(M,g)$$ coincides. In other words the $L^p$-Stokes theorem holds on $(M,g)$ at the level of $(2m-1)$-forms. As $(M,g)$ has finite volume and $q\geq p$ the conclusion follows immediately by Prop. \ref{LpStokes}.
\end{proof}

\begin{prop}
\label{cohomology}
Let $(M,\omega,J,g)$ be a possibly incomplete almost K\"ahler manifold of finite volume and dimension $2m>0$. Assume that $$\dim(\overline{H}^{2m}_{p,\max}(M,g))=1$$  for some $1<p<\infty$. Then $\omega^k$ induces a non trivial class in $\overline{H}^{2k}_{q,\max}(M,g)$ for any $k=1,2,...,m$ and for any $q\in [p,\infty]$.
\end{prop}

\begin{proof}
Since $\overline{H}^{2m}_{p,\max}(M,g)=L^p\Omega^{2m}(M,g)/\overline{\im(d_{2m-1,p,\max})}$ and $\overline{H}^{2m}_{p,\min}(M,g)=L^p\Omega^{2m}(M,g)/\overline{\im(d_{2m-1,p,\min})}$ the identity $L^p\Omega^{2m}(M,g)\rightarrow L^p\Omega^{2m}(M,g)$ induces an injective map 
\begin{equation}
\label{iniezione}
\overline{H}^{2m}_{p,\max}(M,g)\hookrightarrow \overline{H}^{2m}_{p,\min}(M,g).
\end{equation}
As $(M,g)$ is connected and has finite volume we have  $\dim(\overline{H}^{0}_{p',\max}(M,g))=1$ where $p'=p/(p-1)$. This latter property follows immediately by the fact that if $f\in \mathcal{D}(d_{0,\max,p'})$ and $d_{0,\max,p'}f=0$ then $f\in C^{\infty}(M)$ and therefore it has to be constant, see e.g. \cite{GoTr} Th. 12.5. As $M$ is connected and  $\vol_g(M)<\infty$ we can conclude that $\ker(d_{0,p',\max})=\mathbb{R}$.  Using the duality \eqref{stock} this in turn implies that  $\dim(\overline{H}^{2m}_{p,\min}(M,g))=1$. Thus, thanks to \eqref{iniezione} and the fact that $\overline{H}^{2m}_{p,\max}(M,g)=L^p\Omega^{2m}(M,g)/\overline{\im(d_{2m-1,p,\max})}$ and $\overline{H}^{2m}_{p,\min}(M,g)=L^p\Omega^{2m}(M,g)/\overline{\im(d_{2m-1,p,\min})}$ we can conclude that $$\overline{\im(d_{2m-1,p,\max})}=\overline{\im(d_{2m-1,p,\min})}.$$ Since $q\geq p$ and $(M,g)$ has finite volume, the conclusion now follows by Th. \ref{calamari}.
\end{proof}

So far we deduced only non-vanishing properties for the class $[\omega^k]$ with respect to the reduced $L^p$-maximal cohomology. Requiring $(M,g)$ to satisfy some stronger assumption, for instance  $q$-parabolicity, we can also show the non-vanishing of $[\omega^k]$ with respect to the reduced $L^p$-minimal cohomology. This is the goal of the next:

\begin{teo}
\label{qpara}
Let $(M,\omega,J,g)$ be a possibly incomplete almost K\"ahler manifold of finite volume and dimension $2m>0$. Assume that $(M,g)$ is $p$-parabolic with $1<p<\infty$ and let $p'=p/(p-1)$. 
\begin{enumerate}
\item For any $k=1,2,...,m$, $\omega^k$ induces a non trivial class in    $\overline{H}^{2k}_{q,\max}(M,g)$ for any  $q\in [p',\infty]$.
\item For any $k=1,2,...,m$, $\omega^k$ induces a non trivial class in    $\overline{H}^{2k}_{q,\min}(M,g)$ for any  $q\in [1,p]$.
\end{enumerate}
\end{teo}

\begin{proof}
According to Prop. \ref{paraSob} we know that $W^{1,p}(M,g)=W^{1,p}_0(M,g)$. This last equality amounts to saying that the operators $$d_{0,\max,p}:L^p(M,g)\rightarrow L^p\Omega^1(M,g)\ \mathrm{and}\  d_{0,\min,p}:L^p(M,g)\rightarrow L^p\Omega^1(M,g)$$ coincide. Using  \eqref{cope} and \eqref{naghen} we get that also the following two operators $$d_{2m-1,\max,p'}:L^{p'}\Omega^{2m-1}(M,g)\rightarrow L^{p'}\Omega^{2m}(M,g)\ \mathrm{and}\ d_{2m-1,\min,p'}:L^{p'}\Omega^{2m-1}(M,g)\rightarrow L^{p'}\Omega^{2m}(M,g)$$ coincides. Hence the $L^{p'}$-Stokes theorem holds on $(M,g)$ at the level of $(2m-1)$-forms. As $(M,g)$ has finite volume and $q\geq p'$ the thesis is now a consequence of Prop. \ref{LpStokes}. Now we deal with the {\em second point}. 
First we have to  show that $\omega^k\in \ker(d_{k,\min,q})$ for each $k=1,...,m$ and each $q\in [1,p]$. This follows easily using  the fact that $(M,g)$ is $p$-parabolic. Let us consider a sequence of Lipschitz functions with compact supports,  $\{\phi_j\}$, that makes $(M,g)$ $p$-parabolic. As $(M,g)$ has finite volume it is straightforward to check that $\{\phi_j\}$  makes $(M,g)$ $q$-parabolic for any $q\in [1,p]$. Consider the sequence $\{\phi_j\omega^k\}$. Then, according to Prop. \ref{tecprop}, we know that $\{\phi_j\omega^k\}\subset \mathcal{D}(d_{2k,\max,q})$ and $d_{2k,\max,q}(\phi_j\omega^k)=\phi_jd_{2k}\omega^k+d_0\phi_j\wedge \omega^k$ $=d_0\phi_j\wedge \omega^k$  for any $q\in [1,p]$. It is now immediate to verify that $\phi_j\omega^k\rightarrow \omega^k$ as $j\rightarrow \infty$ in $\mathcal{D}(d_{2k,\max,q})$ with respect to the corresponding graph norm. As $\phi_j\omega^k$ has compact support for each $j\in \mathbb{N}$, we can use Prop. \ref{olimpia} to deduce that $\omega^k\in \mathcal{D}(d_{2k,\min,q})$. Finally, as $\omega^k\in \mathcal{D}(d_{2k,\min,q})\cap \ker(d_{2k,\max,q})$ we can conclude that $\omega^k\in \ker(d_{2k,\min,q})$ for any $q\in [1,p]$ as required. Thus we know that $\omega^k$ induces a class in $\overline{H}^{2k}_{q,\min}(M,g)$. In order to complete the proof we have to show that $\omega^k\not \in \overline{\im(d_{2k-1,q,\min})}$ for any $q\in [1,p]$ and $k=1,2,...,m$. By contrast let's assume that $\omega^k \in \overline{\im(d_{2k-1,q,\min})}$ for some $q\in[1,p]$ and $k=1,2,...,m$. Thus there exists a sequence $\{\beta_n\}\subset \Omega_c^{2k-1}(M)$ such that $d_{2k-1}\beta_n\rightarrow \omega^k$ in $L^q\Omega^{2k}(M,g)$. Clearly $\{\beta_n\wedge \omega^{m-k}\}\subset \Omega_c^{2m-1}(M)$ and $d_{2k-1}(\beta_n\wedge \omega^{m-k})\rightarrow \omega^m$ in $L^q\Omega^{2m}(M,g)$ as $n\rightarrow \infty$. As $\vol_g(M)<\infty$ we can deduce that $d_{2k-1}(\beta_n\wedge \omega^{m-k})\rightarrow \omega^m$ in $L^1\Omega^{2m}(M,g)$ as $n\rightarrow \infty$. Hence, using Prop. \ref{littletool}, we can conclude that $\vol_g(M)=0$ which is clearly absurd. In conclusion we proved that $\omega^k$ induces a non-trivial class in $\overline{H}^{2k}_{q,\min}(M,g)$ for any $q\in[1,p]$ and $k=1,2,...,m$ as desired.

\end{proof}

\begin{prop}
\label{Sobolev}
Let $(M,\omega,J,g)$ be a possibly incomplete almost K\"ahler manifold of finite volume and  dimension $2m>0$. Assume that $W^{1,p}(M,g)=W^{1,p}_0(M,g)$ for some $1<p<\infty$ and let $p'=p/(p-1)$. 
\begin{enumerate}
\item For any $k=1,2,...,m$, $\omega^k$ induces a non trivial class in    $\overline{H}^{2k}_{q,\max}(M,g)$ for any  $q\in [p',\infty]$.
\item For any $k=1,2,...,m$, $\omega^k$ induces a non trivial class in    $\overline{H}^{2k}_{q,\min}(M,g)$ for any  $q\in [1,p]$.
\end{enumerate}
\end{prop}

\begin{proof}
This follows immediately by Prop. \ref{paraSob} and Th. \ref{qpara}.
\end{proof}

\begin{prop}
\label{stoc}
Let $(M,\omega,J,g)$ be a possibly incomplete almost K\"ahler manifold of finite volume and  dimension $2m>0$.  Assume that $(M,g)$ is stochastically complete. 
\begin{enumerate}
\item For any $k=1,2,...,m$, $\omega^k$ induces a non trivial class in    $\overline{H}^{2k}_{q,\max}(M,g)$ for any  $q\in [2,\infty]$.
\item For any $k=1,2,...,m$, $\omega^k$ induces a non trivial class in    $\overline{H}^{2k}_{q,\min}(M,g)$ for any  $q\in [1,2]$.
\end{enumerate}
\end{prop}

\begin{proof}
Since $(M,g)$ is stochastically complete we know that $W^{1,2}_0(M,g)=W^{1,2}(M,g)$, see e.g. \cite{GrMa} Th. 1.7. In particular, as $(M,g)$ has finite volume, stochastic completeness actually amounts to the equality $W^{1,2}_0(M,g)=W^{1,2}(M,g)$ which in turn is equivalent to requiring $(M,g)$ being  parabolic, see for instance \cite{BeGu} Prop. 2.6 or \cite{masa}. Finally the conclusion follows by Th. \ref{qpara}.
\end{proof}

We have now some corollaries. First we need to introduce some notations. According to \eqref{corno}-\eqref{delcatria} we know that for any $1\leq p\leq q\leq \infty$ the inclusion $i:\mathcal{D}(d_{k,\max/\min,q})\hookrightarrow \mathcal{D}(d_{k,\max/\min,p})$ induces a natural map at the level of cohomology groups $i_*:\overline{H}^{2k}_{q,\max/\min}(M,g)\rightarrow\overline{H}^{2k}_{p,\max/\min}(M,g)$. In the next statements the image of $\overline{H}^{2k}_{q,\max/\min}(M,g)$ in $\overline{H}^{2k}_{p,\max/\min}(M,g)$ through $i_*$ will be denoted by 
$\im\left(\overline{H}^{2k}_{q,\max/\min}(M,g)\rightarrow\overline{H}^{2k}_{p,\max/\min}(M,g)\right).$

\begin{cor}
\label{primo}
In the setting of Th. \ref{calamari}. 
\begin{enumerate}
\item For any $p\leq q\leq s\leq\infty$ and  $k=1,2,...,m$ we have 
$$\im\left(\overline{H}^{2k}_{s,\max}(M,g)\rightarrow\overline{H}^{2k}_{q,\max}(M,g)\right) \neq \{0\}.$$
\item  Assume now that $1<p<\infty$. For any $1<r\leq p'$, with $p'=p/(p-1)$, and  $k=1,2,...,m$ we have $\overline{H}^{2k}_{r,\min}(M,g)\neq \{0\}$.
\end{enumerate}
\end{cor}

\begin{proof}
The first point follows by the fact that, as showed in Th. \ref{calamari}, $\omega^k$ induces a non trivial class in $H^{2k}_{q,\max}(M,g)$ for any $q\in [p,\infty]$ and $k=1,2,...,m$.  The second point follows by Th. \ref{calamari} and \eqref{stock}. 
\end{proof}

\begin{cor}
\label{secondo}
Let $(M,\omega,J,g)$ be a possibly incomplete almost K\"ahler manifold of finite volume and complex dimension $m>0$. Assume that one among Propositions \ref{LpStokes}, \ref{Lpdiver}, \ref{cohomology} holds true for $(M,g)$. Then $$\im\left(\overline{H}^{2k}_{s,\max}(M,g)\rightarrow\overline{H}^{2k}_{q,\max}(M,g)\right) \neq \{0\}$$ for any $p\leq q\leq s\leq\infty$ and  $k=1,2,...,m$  and 
$$\overline{H}^{2k}_{r,\min}(M,g)\neq \{0\}$$ for any $1<r\leq p'$, with $p'=p/(p-1)$, and  $k=1,2,...,m$.
\end{cor}

\begin{proof}
This is an immediate application of Cor \ref{primo} and Prop. \ref{LpStokes}, \ref{Lpdiver}, \ref{cohomology}.
\end{proof}

\begin{cor}
\label{terzo}
In the setting of Th. \ref{qpara}.
\begin{enumerate}
\item For any $p'\leq q\leq s\leq\infty$ and  $k=1,2,...,m$ we have 
$$\im\left(\overline{H}^{2k}_{s,\max}(M,g)\rightarrow\overline{H}^{2k}_{q,\max}(M,g)\right) \neq \{0\}.$$
\item  For any $1\leq q\leq p$ and  $k=1,2,...,m$ we have 
$$\im\left(\overline{H}^{2k}_{\infty}(M,g)\rightarrow\overline{H}^{2k}_{q,\min}(M,g)\right) \neq \{0\}.$$
This in turn implies that for any $1\leq q\leq s\leq p$ and  $k=1,2,...,m$ we have$$\im\left(\overline{H}^{2k}_{s,\min}(M,g)\rightarrow\overline{H}^{2k}_{q,\min}(M,g)\right) \neq \{0\}.$$
\item Assume that $p\geq 2$. For any $p'\leq q\leq p$ and  $k=1,2,...,m$ we have 
$$\im\left(\overline{H}^{2k}_{q,\min}(M,g)\rightarrow\overline{H}^{2k}_{q,\max}(M,g)\right) \neq \{0\}.$$
\end{enumerate}
\end{cor}

\begin{proof}
The first property follows by arguing as in Cor. \ref{primo}. Concerning the second point we first note that, thanks to Prop. \ref{messico} and the fact that $\vol_g(M)<\infty$, we have a continuous inclusion $\mathcal{D}(d_{k,\infty})\hookrightarrow \mathcal{D}(d_{k,z,\min})$ for any $k=0,1,...,2m$ and $z\in [1,\infty)$ where each domain is endowed with the corresponding graph norm. As previously recalled this inclusion induces a natural map at the level of cohomology groups $\overline{H}^{2k}_{\infty}(M,g)\rightarrow\overline{H}^{2k}_{q,\min}(M,g)$. Since we have shown that $0\neq [\omega^k]\in \overline{H}^{2k}_{\infty}(M,g)$ for any  $k=1,2,...,m$ and $0\neq [\omega^k]\in \overline{H}^{2k}_{q,\min}(M,g)$ for any $q\in [1,p]$ and $k=1,2,...,m$ we can conclude that $\im(\overline{H}^{2k}_{\infty}(M,g)\rightarrow\overline{H}^{2k}_{q,\min}(M,g)) \neq \{0\}$ for any $1\leq q\leq p$ and  $k=1,2,...,m$. The same argument leads to the conclusion that  for any $k=1,2,...,m$ we have  $\im(\overline{H}^{2k}_{s,\min}(M,g)\rightarrow\overline{H}^{2k}_{q,\min}(M,g)) \neq \{0\}$ for each $1\leq q\leq s\leq p$. Finally in the case $p\geq 2$, as we know that both $0\neq [\omega^k]\in \overline{H}^{2k}_{q,\min}(M,g)$ and $0\neq [\omega^k]\in \overline{H}^{2k}_{q,\max}(M,g)$ for any $k=1,2,...,m$ and  $p'\leq q\leq p$, we can conclude that
$\im(\overline{H}^{2k}_{q,\min}(M,g)\rightarrow\overline{H}^{2k}_{q,\max}(M,g)) \neq \{0\}$ for each $p'\leq q\leq p$ and $k=1,2,...,m$.
\end{proof}

We have also the following corollaries whose proofs are omitted because straightforward.

\begin{cor}
\label{quarto}
In the setting of Prop. \ref{Sobolev}. Then Cor. \ref{terzo} holds true for $(M,g)$.
\end{cor}

\begin{cor}
\label{quinto}
In the setting of Prop. \ref{stoc}. Then  Cor. \ref{terzo} holds true for $(M,g)$ with $p=p'=2$. 
\end{cor}

\begin{rem}
Thanks to \eqref{surj1} and \eqref{surj2} we know that whenever the reduced $L^p$-maximal/minimal cohomology is not trivial then also the $L^p$-maximal/minimal cohomology is not trivial.
\end{rem}

\section{Examples and applications}
In this section we explore various examples and applications of the previous results. We start with the case of complete almost K\"ahler manifolds with finite volume. First we need to recall the following important property:
\begin{prop}
\label{Lpmaxmin}
Let $(M,g)$ be a complete Riemannian  manifold of dimension $m$. Then for any $p\in[1,\infty)$ and $k=0,1,...,m$, the following two operators 
$$d_{k,\max,p}:L^p\Omega^k(M,g)\rightarrow L^p\Omega^{k+1}(M,g)\ \mathrm{and}\ d_{k,\min,p}:L^p\Omega^k(M,g)\rightarrow L^p\Omega^{k+1}(M,g)$$ coincide.
\end{prop}
\begin{proof}
This is a well known result whose proof, in the case $p=2$, goes back to Gaffney \cite{Gaf}. For the sake of completeness we recall here a proof. It is a well known fact that completeness is equivalent to the existence of a sequence of smooth functions with compact support $\{\phi_n\}\subset C^{\infty}_c(M)$ such that:\\
a) $0\leq \phi_n\leq 1$;\\
b) $\{A_n\}$ is an exhaustion of $M$ made by open subsets with compact closure, where $A_n$ is the interior of $\supp(\phi_n)$, the support of $\phi_n$;\\
c) $\lim d_0\phi_n=0$ in $L^{\infty}\Omega^1(M,g)$ as $n\rightarrow \infty$.\\ For a proof we refer for instance to \cite{Dema} Lemma 12.1 pag. 57. Consider now any form $\eta\in \mathcal{D}(d_{k,\max,p})$. Then, thanks to Prop. \ref{tecprop}, we know that $\phi_n\eta\in \mathcal{D}(d_{k,\max,p})$ and $d_{k,\max,p}\phi_n\eta=d_0\phi_n\wedge \eta+(-1)^k\phi_nd_{k,\max,p}\eta$. By virtue of Lebesgue's dominate convergence theorem, the inequality $\|d_0\phi_n\wedge \eta\|_{L^p\Omega^{k+1}(M,g)}\leq \|d_0\phi_n\|_{L^{\infty}\Omega^1(M,g)}\|\eta\|_{L^p\Omega^{k}(M,g)}$ and the fact that $\lim d_0\phi_n=0$ in $L^{\infty}\Omega^1(M,g)$ as $n\rightarrow \infty$ it is easy to verify that $\phi_n\eta\rightarrow \eta$ in $\mathcal{D}(d_{k,\max,p})$ with respect to the corresponding graph norm as $n\rightarrow \infty$. Finally, as $\phi_n\eta$ has compact support, we can use Prop. \ref{olimpia} to deduce that $\{\phi_n\eta\}\subset \mathcal{D}(d_{k,\min,p})$ which in turn implies that $\eta\in \mathcal{D}(d_{k,\min,p})$ too.
\end{proof}

We can summarize the above proposition by saying that on a complete Riemannian manifold of dimension $m$ the {\em $L^p$-Stokes theorem holds true} for any $p\in [1,\infty)$ and any $k=0,1,2,...,m.$ Clearly the equality $d_{k,\max,p}=d_{k,\min,p}$ implies immediately that the (reduced) $L^p$-maximal cohomology and  (reduced) $L^p$-minimal cohomology coincide, that is $\overline{H}^k_{p,\max}(M,g)=\overline{H}^{k}_{p,\min}(M,g)$. Henceforth we will simply label with $\overline{H}^k_{p}(M,g)$ the (reduced) $L^p$-cohomology of a complete Riemannian manifold $(M,g)$.

\begin{prop}
\label{comple}
Let $(M,\omega,J,g)$ be a complete almost K\"ahler manifold of finite volume. Then $\omega^k$ induces a non-trivial class in both $\overline{H}^{2k}_{\infty,\max}(M,g)$ and $\overline{H}^{2k}_{p}(M,g)$ for any $k=1,2,...,m$ and $p\in[1,\infty)$. Moreover Corollaries  \ref{primo}--\ref{quinto} hold true for $(M,g)$.
\end{prop}

\begin{proof}
Thanks to Prop. \ref{Lpmaxmin} we know that on $(M,g)$ the $L^p$-Stokes theorem holds true for any $p\in [1,\infty)$ and any $k=0,1,2,...,2m$. Now, thanks to Prop. \ref{LpStokes}, we can conclude that $\omega^k$ induces a non-trivial class in both $\overline{H}^{2k}_{\infty,\max}(M,g)$ and $\overline{H}^{2k}_{p}(M,g)$ for any $k=1,2,...,m$ and $p\in[1,\infty)$. 
\end{proof}

In order to provide some examples to Prop. \ref{comple} we recall now  the definition of various {\em complete non-compact K\"ahler manifolds with finite volume} that already appeared in the literature. First we recall some notions that will be needed to describe these examples.\\  

{\em Complex spaces} are a classical topic in complex geometry and we refer to \cite{GeFi} and to \cite{GrRe} for definitions and properties. Here we recall only that a reduced complex space is {\em irreducible} if and only if $\reg(X)$ is connected. A reduced and paracompact complex space $X$ is said  \emph{Hermitian} if the regular part  $\reg(X):=X\setminus \sing(X)$ carries a Hermitian metric $h$ such that for every point $x\in X$ there exists an open neighborhood $U\ni p$ in $X$, a proper holomorphic embedding of $U$ into a polydisc $\phi: U \rightarrow \mathbb{D}^N\subset \mathbb{C}^N$ and a Hermitian metric $\beta$ on $\mathbb{D}^N$ such that $(\phi|_{\reg(U)})^*\beta=h$. If $d\omega=0$, that is the fundamental form of $h$ is closed, then $(X,h)$ will be called a {\em K\"ahler space}. Important examples of K\"ahler spaces are given by {\em complex projective varieties} $V\subset \mathbb{C}\mathbb{P}^n$ with the K\"ahler metric on $\reg(V)$ induced by the Fubini-Study metric of $\mathbb{C}\mathbb{P}^n$.  More generally, given a compact K\"ahler manifold $(M,g)$, any analytic subvariety $X\subset M$ whose regular part carries the K\"ahler metric induced by $g$ is a compact K\"ahler space. We refer to \cite{ToH} for more details.
In addition we recall that two Riemannian metrics $g_1$ and $g_2$ over a manifold $M$ are said {\em quasi-isometric} if there exist positive constants $a$ and $b$ such that $ag_1\leq g_2\leq b g_1$.\\

\textbf{Poincar\'e-type K\"ahler metrics.} Let $(M,h)$ be a compact K\"ahler manifold  with fundamental form $\omega$. Let $D$ be a normal crossing divisor. Let $L_D$ be the line bundle on $M$ associated to $D$.  Let us label by $s:M\rightarrow L_D$ a global holomorphic section whose  associated divisor $(s)$ equals $D$. Let $\rho$ be any Hermitian metric on $L_D$ such that $\|s\|_{\rho}$, the norm of $s$ with respect to $\rho$, satisfies $\|s\|_{\rho}<1$. A K\"ahler  metric $g$ on $M\setminus D$ which is quasi-isometric to a K\"ahler metric with fundamental $(1,1)$-form $$b\omega-\frac{\sqrt{-1}}{2\pi}\partial\overline{\partial}\log(\log\|s\|_{\rho}^2)^2$$ for $b$ a positive integer, will be called a {\em Poincar\'e-type metric}. Endowed with $g$, $(M\setminus D)$ becomes a complete non-compact K\"ahler  manifold with finite volume. We refer to \cite{CoGr},  \cite{GMMI} and  \cite{SZu} for more on these metrics.\\

\textbf{Saper-type K\"ahler metrics}. 
 Let $X$ be a singular subvariety of a compact K\"ahler manifold $(M,h)$. Let $\omega$ be the fundamental $(1,1)$-form of $(M,h)$. Let $\pi:N\rightarrow M$ be a holomorphic map from a compact complex manifold $N$ to $M$ whose exceptional set $D$ is a divisor with normal crossings in $N$ and such that the restriction $$\pi|_{N\setminus D}:N\setminus D \longrightarrow M\setminus \sing(X)$$ is a biholomorphism. Let $L_D$ be the line bundle on $N$ associated to $D$. Let $s:N\rightarrow L_D$ be a global holomorphic section whose  associated divisor $(s)$ equals $D$. Let $\tau$ be any Hermitian metric on $L_D$ such that $\|s\|_{\tau}$, the norm of $s$ with respect to $\tau$, satisfies $\|s\|_{\tau}<1$.  A K\"ahler  metric $g_S$ on $N\setminus D$ which is quasi-isometric to a K\"ahler metric with fundamental $(1,1)$-form $$l\pi^*\omega-\frac{\sqrt{-1}}{2\pi}\partial\overline{\partial}\log(\log\|s\|_{\tau}^2)^2$$ for $l$ a positive integer, will be called a {\em Saper-type K\"ahler metric}, distinguished with respect to the map $\pi$. The corresponding K\"ahler metric on $M\setminus {\sing{X}}\cong N\setminus D$ and its restriction to $X\setminus \sing{X}$ is also called Saper-type K\"ahler metric. Endowed with $g_S$ the complex manifolds $N\setminus D$, $M\setminus \sing(X)$ and $\reg(X)$ become complete non-compact K\"ahler manifolds of finite volume. These metrics were introduced by Saper in \cite{LeSa} in the setting of complex projective varieties  with isolated singularities. Their construction was later generalized  by Grant-Melles and Milman in \cite{GMMI} and \cite{GMMIL} to the case of an arbitrary subvariety of a compact K\"ahler manifold. The above definition is taken from \cite{GMMIL}.\\

\textbf{Grauert-type K\"ahler metrics}. Let $(X,h)$ be a K\"ahler space and let $\omega$ be the fundamental form of $h$.  A {\em Grauert-type K\"ahler metric} is a K\"ahler metric on $\reg(X)$ whose fundamental form is given by $\omega+i\partial\overline{\partial}f$, where $f:X\rightarrow \mathbb{R}$ is  a {\em Grauert potential}.  Equipped with $\omega+i\partial\overline{\partial}f$, $\reg(X)$ becomes a complete non-compact K\"ahler manifold with finite volume. We refer to \cite{ToH} for more details and existence results.\\

\textbf{Bergman metric}. This is the K\"ahler metric on the regular part of the Baily-Borel-Satake compactification of a quotient like $\Gamma\setminus D$ where $D=G/K$ is required to be a bounded symmetric domain, $G$ is the set of real points of a semi-simple algebraic group defined over $\mathbb{Q}$, $K$ a maximal compact subgroup and $\Gamma\subset G$ is an arithmetic subgroup acting freely on $D$. We refer to \cite{SaSt} and the references therein for more on this topic.\\
 
We give now some examples of {\em incomplete  K\"ahler/almost K\"ahler  metrics} fulfilling some of the assumptions used in the previous sections.\\

\textbf{Compact K\"ahler spaces}. As recalled above a K\"ahler space $(X,h)$ is a Hermitian space such that $\omega$, the fundamental form associated to $h$, satisfies $d\omega=0$.  Compact K\"ahler spaces (and more generally compact Hermitian spaces) have finite volume and are $q$-parabolic for any $q\in [1,2]$, see \cite{BeGu} and \cite{JRU}. Therefore we can conclude that Th. \ref{qpara}, Prop. \ref{Sobolev}, \ref{stoc} and the corresponding corollaries hold for $(X,h)$.\\

\textbf{Almost K\"ahler pseudometrics}. Let $(M,J)$ be a compact almost complex manifold. Let $g$ be a section of $T^*M\otimes T^*M\rightarrow M$ such that $g$ is symmetric, non-negative, compatible with $J$ and $g|_A$ is strictly positive, where $A$ is an open and dense subset of $M$. Clearly $A$ has finite volume with respect to $g$. This follows easily by the fact that $M$ is compact and $\dvol_g\leq c\dvol_h$ where $h$ is any Hermitian metric on $M$ and $c>0$ is a suitable constant.
Furthermore assume that $\omega(\ ,\ ):=g(J\ ,\ )$ is closed and that $M\setminus A\subset \cup_{i=1}^s N_i$, where each $N_i$ is a compact submanifold of $M$ satisfying $\text{cod}(N_i)\geq 2$ for every $i=1,...,s$ and where $\text{cod}(S_i)$ is the codimension of $S_i$. Then, according to \cite{BeGu} Prop. 4.7, we know that $(A,g|_A)$ is $q$-parabolic for any $q\in [1,2]$. Hence we can conclude that  Th. \ref{qpara}, Prop. \ref{Sobolev}, \ref{stoc} and the corresponding corollaries hold for $(A,g|_A)$.\\
 An important family of  examples belonging to the above setting can be constructed as follows: Let $(X,h)$ be a compact K\"ahler space of complex dimension $m$. Let $\pi:M\rightarrow X$ be a resolution of $X$, that is $M$ is a compact complex manifold, $\pi:M\rightarrow X$ is holomorphic and surjective, $\pi^{-1}(\sing(X))=D$ is a normal crossing divisor of $M$ and $\pi|_{M\setminus D}:M\setminus D\rightarrow \reg(X)$ is a biholomorphism, see \cite{Hiro}. We recall that a divisor with only normal crossings is a divisor of the form $D=\sum_{i=1}^r V_i$ where $V_i$ are distinct irreducible smooth analytic hypersurfaces of $M$ and $D$ is defined in a neighborhood of any point by an equation in local analytic coordinates of the type $z_1 \cdot \cdot \cdot z_k = 0$, $1\leq k \leq m$. Finally, if we define $g:=\pi^*h$, it is easy to check that  $g$ is a symmetric non-negative section of $T^*M\otimes T^*M\rightarrow M$ that is  compatible with $J$, the complex structure of $M$, and such that  $g|_{(M\setminus D)}$ is strictly positive. In other word $g$ is a {\em K\"ahler pseudometric} on $M$.\\

Now we continue by showing some topological applications of our results. 
In the next proposition the first two points  are  a particular case of a more general non-vanishing theorem that  can be proved as a consequence of the Hard Lefschetz theorem and the decomposition theorem, see e.g. \cite{KiWoo} and the references therein. Here we provide a new proof  based on the results of the previous section. First we need to introduce some notations. In the sequel  $H^{p,q}_{c,\overline{\partial}}(\reg(X))$ will denote the Dolbeault cohomology of $\reg(X)$ with compact support while $H^{p,q}_{2, \overline{\partial}_{\max}}(\reg(X),h)$ will denote the $L^2$-maximal Dolbeault cohomology of $(\reg(X),h)$. Clearly the latter is the cohomology of the $L^2$-maximal Dolbeault complex whose definition is omitted since it is completely analogous to that of the $L^2$-maximal de  Rham complex. We refer to \cite{bovo} for more details.

\begin{prop}
\label{nonvan}
Let $(X,h)$ be a compact and irreducible K\"ahler space of complex dimension $m>0$. Assume that $\dim(\sing(X))=0$. Then:
\begin{enumerate}
\item If $m$ is odd we have $H^{2k}(\reg(X),\mathbb{R})\neq 0$ and $H^{2k}_c(\reg(X),\mathbb{R})\neq 0$ for each $k=0,...,[m/2]$ and $k=[m/2]+1,...,m$, respectively.
\item If $m$ is even we have $H^{2k}(\reg(X),\mathbb{R})\neq 0$ for each $k=0,...,(m/2)-1$, $\im(H^m_c(\reg(X),\mathbb{R})\rightarrow H^m(\reg(X),\mathbb{R}))\neq \{0\}$ for $2k=m$ and  $H^{2k}_c(\reg(X),\mathbb{R})\neq 0$ for each  $k=(m/2)+1,...,m$. 
\item $H^{p,p}_{\overline{\partial}}(\reg(X))\neq 0$ and $H^{p,p}_{c,\overline{\partial}}(\reg(X))\neq 0$ whenever $2p<m-1$ and $2p>m+1$, respectively.
\end{enumerate}
\end{prop}

\begin{proof}
According to \cite{TOh}, \cite{ToH} and  \cite{LeSa} we know that 
\begin{equation}
\label{filoraso}
H^{k}_{2,\max}(\reg(X),h)\cong \left\{
\begin{array}{ll}
H^k(\reg(X),\mathbb{R})\ &\ k\leq m-1\\
\im(H_c^{m}(\reg(X),\mathbb{R})\rightarrow H^k(\reg(X),\mathbb{R}))\ & k=m\\
H_c^k(\reg(X),\mathbb{R})\ &\ k\geq m+1
\end{array}
\right.
\end{equation}
As previously remarked we know that $(\reg(X),h)$ is $q$-parabolic for each $q\in [1,2]$ and with finite volume. Thus we can apply Th. \ref{qpara} in order to conclude that $H^{2k}_{2,\max}(\reg(X),h)\neq \{0\}$ for each $k=1,...,m$. Moreover $H^{0}_{2,\max}(\reg(X),h)\neq \{0\}$ too as $H^{0}_{2,\max}(\reg(X),h)=\ker(d_{0,\max})=\mathbb{R}$, as a consequence of the fact that $(\reg(X),h)$ is connected and with finite volume. This together with \eqref{filoraso} shows that the first two properties of the above proposition hold true. Concerning the third point we first claim that 
\begin{itemize}
\item $\omega^k$ induces a non trivial class in $H^{k,k}_{2,\overline{\partial}_{\max}}(\reg(X),h)$ for each $k=1,...,m$.
\end{itemize}
By contrast let us assume the existence of an integer $k\in \{1,2,...,m\}$ and a $(k,k-1)$-form $\eta\in \mathcal{D}(\overline{\partial}_{k,k-1,\max})$ such that $\overline{\partial}_{k,k-1,\max}\eta=\omega^k$. As $\Omega^{k,k-1}(\reg(X))\cap \mathcal{D}(\overline{\partial}_{k,k-1,\max})$ is dense in $\mathcal{D}(\overline{\partial}_{k,k-1,\max})$ with respect to the corresponding graph norm, see for instance \cite{BL} pag. 98, we can deduce the existence of a sequence of smooth forms $\{\eta_n\}\subset \Omega^{k,k-1}(\reg(X))\cap \mathcal{D}(\overline{\partial}_{k,k-1,\max})$ such that $$\lim_{n\rightarrow \infty}\overline{\partial}_{k,k-1}\eta_n=\omega^k$$ in $L^2\Omega^{k,k}(\reg(X),h)$. 
Let us consider  now the sequence $\{\eta_n\wedge \omega^{m-k}\}$. Clearly $$\{\eta_n\wedge \omega^{m-k}\}\subset \Omega^{2m-1}(\reg(X))\cap L^2\Omega^{2m-1}(\reg(X),h)$$ as $\eta_n\in L^2\Omega^{k,k-1}(\reg(X),h)\cap \Omega^{k,k-1}(\reg(X))$ and $\omega^{m-k}\in$  $L^{\infty}\Omega^{m-k,m-k}(\reg(X),h)\cap \Omega^{m-k,m-k}(\reg(X))$. Concerning $\overline{\partial}_{2m,2m-1}(\eta_n\wedge \omega^{m-k})$ we have $$\overline{\partial}_{2m,2m-1}(\eta_n\wedge \omega^{m-k})=(\overline{\partial}_{k,k-1}\eta_n)\wedge \omega^{m-k}$$ which certainly  lies in $L^2\Omega^{m,m}(\reg(X),h)\cap \Omega^{m,m}(\reg(X))$ since $\overline{\partial}_{k,k-1}\eta_n\in L^2\Omega^{k,k}(\reg(X),h)\cap \Omega^{k,k}(\reg(X))$ and  
$\omega^{m-k}\in \Omega^{m-k,m-k}(\reg(X))\cap L^{\infty}\Omega^{m-k,m-k}(\reg(X),h)$. In particular $$\lim_{n\rightarrow \infty}\overline{\partial}_{m,m-1}(\eta_n\wedge \omega^{m-k})=\omega^m$$ in $L^2\Omega^{m,m}(\reg(X),h)$.
Moreover $\partial_{m,m-1}(\eta_n\wedge \omega^{m-k})=0$ as $\eta_n\wedge \omega^{m-k}\in \Omega^{m,m-1}(\reg(X))$. This in turn tells us that $$d_{2m-1}(\eta_n\wedge\omega^{m-k})=\overline{\partial}_{m,m-1}(\eta_n\wedge \omega^{m-k})\in L^2\Omega^{m,m}(\reg(X),h).$$ Therefore $\{\eta_n\wedge \omega^{m-k}\}\subset L^2\Omega^{2m-1}(\reg(X),h)\cap \Omega^{2m-1}(\reg(X))$, $\{d_{2m-1}(\eta_n\wedge \omega^{m-k})\}\subset L^2\Omega^{2m}(\reg(X),h)\cap \Omega^{2m}(\reg(X))$, that is $\{\eta_n\wedge \omega^{m-k}\}\in \mathcal{D}(d_{2m-1,\max})$. Furthermore we know that $(\reg(X),h)$ is parabolic and so, as explained in the proof of Th. \ref{qpara}, we have $d_{0,\max}=d_{0,\min}$. Taking the adjoint and composing with the Hodge star operator we reach the conclusion that $d_{2m-1,\min}=d_{2m-1,\max}$. Eventually we showed that $\{\eta_n\wedge \omega^{m-k}\}\in \mathcal{D}(d_{2m-1,\min})$ and $$\lim_{n\rightarrow \infty}d_{2m-1}(\eta_n\wedge \omega^{m-k})=\omega^m.$$ As $\vol_h(\reg(X))<\infty$ we can now use Prop. \ref{littletool} to conclude that $\vol_h(\reg(X))=0$, which is clearly absurd. Hence we showed that $\omega^k$ induces a non trivial class in $H^{k,k}_{2,\overline{\partial}_{\max}}(\reg(X),h)$ for each $k=1,...,m$. The conclusion now follows by the fact that $H^{p,q}_{2,\overline{\partial}_{\max}}(\reg(X),h)\cong H^{p,q}_{\overline{\partial}}(\reg(X))$ for $p+q<m-1$ and $H^{p,q}_{2,\overline{\partial}_{\max}}(\reg(X),h)\cong H^{p,q}_{c,\overline{\partial}}(\reg(X))$ for $p+q>m-1$, see \cite{ToH} Th. 2.30 and Th. 2.31.
\end{proof}

Finally we come to the main topological application. In order to prove the next theorem we need to introduce the following variant of the $L^{\infty}$-cohomology. Let $(M,g)$ be an arbitrary Riemannian manifold.  Let us define $$\mathrm{C}\Omega^k_{\infty}(M,g):=\{ \omega\in \mathcal{D}(d_{k,\max,\infty})\cap \mathrm{C}\Omega^k(M)\ \text{such that}\ d_{k,\max,\infty}\omega \in \text{C}\Omega^{k+1}(M)\}$$ where $\mathrm{C}\Omega^k(M)$ and $\mathrm{C}\Omega^{k+1}(M)$ are the spaces of continuous $k$-form and $(k+1)$-forms over $M$, respectively. Consider now the complex $$...\stackrel{d_{k-1,\max,\infty}}{\rightarrow}\mathrm{C}\Omega^k_{\infty}(M,g)\stackrel{d_{k,\max,\infty}}{\rightarrow}\mathrm{C}\Omega^{k+1}_{\infty}(M,g)\stackrel{d_{k+1,\max,\infty}}{\rightarrow}$$ and let us label by $H^k_{\infty,\max,\text{c}}(M,g)$ the cohomology of the above complex. We have now all the ingredients for the next

\begin{teo}
\label{singhom}
Let $(X,h)$ be a compact and irreducible K\"ahler space of complex dimension $m>0$. Assume that every point $p\in \sing(X)$ has a local base of open neighborhoods whose regular parts are connected. Then  $$H^{2k}(X,\mathbb{R})\neq  \{0\}$$ for each $k=0,1...,m$.
\end{teo}
\begin{proof}

In order to prove the above theorem we need to show the following property that we believe to have an independent interest.
\begin{itemize}
\item Let $(X,h)$ be as above. Then $H^k_{\infty, \max, \text{c}}(\reg(X),h)\cong H^k(X,\mathbb{R})$ for any $k=0,1,...,2m$.
\end{itemize}
The above claim is essentially a consequence of the main result proved in \cite{Val}. As the setting in \cite{Val} is slightly different than ours we give here all the details. Thanks to the triangulation theorem of Lojasiewicz, see \cite{Loj}, we know that $X$ is locally contractible. This in turn implies, see \cite{Rama} or \cite{Sella}, that the cohomology of $X$ with coefficients in the constant sheaf $\mathbb{R}_{X}$ and the singular cohomology of $X$ with coefficients in $\mathbb{R}$ are isomorphic. Hence during the remaining part of the proof we will use the notation $H^k(X,\mathbb{R})$ for both the singular cohomology of $X$ with coefficients in $\mathbb{R}$ and the cohomology of $X$ with coefficients in the constant sheaf $\mathbb{R}_{X}$. Consider now the presheaf  on $X$ defined by the assignment $$X\supset U\mapsto \{ \text{C}\Omega^k_{\infty}(\reg(U),h|_{\reg(U)})\}$$ where $U$ is any open subset of $X$ and $k=0,...,2m$. In other words this is the presheaf that assigns to every open subset $U\subset X$ the space of continuous $k$-forms lying in $L^{\infty}\Omega^k(\reg(U),h|_{\reg(U)})$ whose image under the distributional action of $d_k$ lies in $\text{C}\Omega^{k+1}(\reg(U))\cap L^{\infty}\Omega^{k+1}(\reg(U),h|_{\reg(U)})$. Let us label by $\mathcal{C}^{\infty}_h\Omega^k$ the corresponding sheaf arising by sheafification. The sections of $\mathcal{C}^{\infty}_h\Omega^k$ over an open subset $U\subset X$ are given by 
\begin{multline}
\label{sheafifi}
\mathcal{C}^{\infty}_h\Omega^k(U)=\{s\in \text{C}\Omega^k(\reg(U))\cap L^{\infty}_{\mathrm{loc}}\Omega^{k}(\reg(U),h|_{\reg(U)})\ \text{such that for each }\ p\in U\ \text{there exists an}\\ \text{ open}\ \text{neighborhood}\ W\ \text{with}\ p\in W\subset\ U\ \text{such that}\ s|_{\reg(W)}\in  \text{C}\Omega_{\infty}^k(\reg(W),h|_{\reg(W)})\}.
\end{multline}
Clearly the differential $d_k$ induces a sheaves morphism that for simplicity we still label by $d_k:\mathcal{C}^{\infty}_h\Omega^k\rightarrow \mathcal{C}^{\infty}_h\Omega^{k+1}$ and that obeys $d_{k+1}\circ d_k=0$. We claim now that the following complex of sheaves over $X$
\begin{equation}
\label{resolution}
0\rightarrow\mathcal{C}^{\infty}_h\Omega^0\stackrel{d_0}{\rightarrow}...\stackrel{d_{k-1}}{\rightarrow}\mathcal{C}^{\infty}_h\Omega^k\stackrel{d_{k+1}}{\rightarrow}...\stackrel{d_{2m-1}}{\rightarrow}\mathcal{C}^{\infty}_h\Omega^{2m}\rightarrow 0
\end{equation} 
is a fine resolution of the constant sheaf $\mathbb{R}_{X}$. By the assumptions we know that  every point $p\in \sing(X)$ has a local base of open neighborhoods whose regular parts are connected. Now it is immediate to verify  that $\ker(\mathcal{C}^{\infty}_h\Omega^0\stackrel{d_0}{\rightarrow}\mathcal{C}^{\infty}_h\Omega^1)$ gives the constant sheaf $\mathbb{R}_{X}$. Moreover the existence of a partition of unity with bounded differential, see \cite{bovo}, assures that the complex \eqref{resolution} is made by fine sheaves. Concerning the exactness of \eqref{resolution}, by the definition of Hermitian complex space, we know that for any $x\in X$ there exists an open neighborhood $U\ni p$ in $X$, a proper holomorphic embedding of $U$ into a polydisc $\phi: U \rightarrow \mathbb{D}^N\subset \mathbb{C}^N$ and a Hermitian metric $\beta$ on $\mathbb{D}^N$ such that $(\phi|_{\reg(U)})^*\beta=h$. Hence we are in position to apply \cite{Val} Th. 4.2.1 to deduce that \eqref{resolution} is an exact sequence of sheaves. As $X$ is compact we have $\mathcal{C}^{\infty}_h\Omega^k(X)=\text{C}\Omega^k_{\infty}(\reg(X),h)$, that is the space of global sections of $\mathcal{C}^{\infty}_h\Omega^k$ equals the space of continuous $k$-forms lying in $L^{\infty}\Omega^k(\reg(X),h)$ whose image under the distributional action of $d_k$ lies in $\text{C}\Omega^{k+1}(\reg(X))\cap L^{\infty}\Omega^{k+1}(\reg(X),h)$. We can thus conclude that the cohomology of the complex given by the global sections: $$0\rightarrow\mathcal{C}^{\infty}_h\Omega^0(X)\stackrel{d_0}{\rightarrow}...\stackrel{d_{k-1}}{\rightarrow}\mathcal{C}^{\infty}_h\Omega^k(X)\stackrel{d_{k+1}}{\rightarrow}...\stackrel{d_{2m-1}}{\rightarrow}\mathcal{C}^{\infty}_h\Omega^{2m}(X)\rightarrow 0$$ coincides with $H^k_{\infty,\max,\text{c}}(\reg(X),h)$ for each $k=0,...,2m$. So we proved that $$H^k(X,\mathbb{R})\cong H^k_{\infty, \max,\text{c}}(\reg(X),h).$$
Clearly $\omega^k$ induces a class in $H^{2k}_{\infty, \max,\text{c}}(\reg(X),h)$ for any $k=1,...,m$. Moreover $0\neq [\omega^k]\in H^{2k}_{\infty, \max,\text{c}}(\reg(X),h)$ because, thanks to Th. \ref{qpara}, we know that $\omega^k\notin \overline{\im(d_{2k-1,\max,\infty})}$. Finally $H^0_{\infty, \max,\text{c}}(\reg(X),h)\neq \{0\}$ because $H^0_{\infty, \max,\text{c}}(\reg(X),h)=C(\reg(X))\cap\ker(d_{0,\max,\infty})=\mathbb{R}$. We can thus conclude that $H^{2k}(X,\mathbb{R})\neq \{0\}$ for each $k=0,1...,m$ as desired. 
\end{proof}

We have now the last result of the paper. In order to state it we recall the notion of {\em normal complex space}.  Let $X$ be a reduced complex space. The sheaf $\widetilde{\mathcal{O}}_X$ of weakly holomorphic functions is the sheaf that assigns to every open subset $U\subset X$ the space of  holomorphic functions $f:U\setminus (U\cap \sing(X))\rightarrow \mathbb{C}$ such that for each point $p\in U$ there exists an open neighborhood $p\in V\subset U$ such that $f|_{V\setminus (\sing(X)\cap V)}$ is bounded. The space $X$ is called normal if $\widetilde{\mathcal{O}}_{X,p}=\mathcal{O}_{X,p}$ for every $p\in X$. See \cite{GeFi} for more details.

\begin{cor}
\label{singhomx}
Let $(X,h)$ be a compact and irreducible normal K\"ahler space of complex dimension $m$. Then $$H^{2k}(X,\mathbb{R})\neq  \{0\}$$ for each $k=0,1,...,m$.
\end{cor}

\begin{proof}
It is easy to see that if $X$ is normal then for each point $p\in X$ there exists a local base of open neighborhoods whose regular parts are connected. Now the above corollary follows by Th. \ref{singhom}.
\end{proof}


\begin{thebibliography}{99}








\bibitem{bovo}
F{.} Bei.
\newblock On the Laplace-Beltrami operator on compact complex spaces.
\newblock https://arxiv.org/abs/1706.05962



\bibitem{FrB}
F{.} Bei,
\newblock Sobolev spaces and Bochner Laplacian on complex projective varieties and stratified pseudomanifolds.
\newblock  \emph{Journal of Geometric Analysis} 27 2017, Issue 1, pp 746--796




\bibitem{BeGu}
F{.} Bei, B{.} Gueneysu.
\newblock $q$-Parabolicity of stratified pseudomanifolds and other singular spaces. 
\newblock {\em Ann. Global Anal. Geom.} 51 (3) (2017) 267--286.






\bibitem {BGV}
N{.} Berline, E{.} Getzler, M{.} Vergne.
\newblock Heat kernels and Dirac operators.
\newblock   Grundlehren Text Editions. Springer-Verlag, Berlin, 2004.



%




\bibitem {BL}
J{.} Br\"uning, M{.} Lesch.
\newblock Hilbert complexes
\newblock  {\em J. Funct. Anal.}, 108 (1992), no. 1, 88--132. 


\bibitem {BLE}
J{.} Br\"uning, M{.} Lesch.
\newblock K\"ahler-Hodge theory for conformal complex cones. 
\newblock {\em Geom. Funct. Anal.},  3 (1993), no. 5, 439--473.




%
%
%


%
%


\bibitem{Cannas}
A{.} Cannas da Silva.
\newblock Lectures on symplectic geometry.
\newblock Lecture Notes in Mathematics 1764, Springer (Berlin) 2001.


\bibitem{JCH}
J{.} Cheeger.
\newblock Spectral geometry of singular Riemannian spaces. 
\newblock {\em J. Differential Geom.} 18 (1983), no. 4, 575--657.








\bibitem{CoGr}
M{.} Cornalba, P{.} Griffiths.
\newblock Analytic cycles and vector bundles on non-compact algebraic varieties.
\newblock {\em Inventh. Math.}, 28 (1975), 1--106.

\bibitem{Dema}
J{.} Bertin, J-P{.} Demailly, L{.} Illusie, C{.} Peters.
\newblock Introduction to Hodge Theory.
\newblock SMF/AMS Texts and Monographs, vol. 8, AMS (2002)


\bibitem{GeFi}
G{.} Fischer.
\newblock Complex analytic geometry. 
\newblock Lecture Notes in Mathematics, vol. 538. Springer, Berlin, New York (1976).

\bibitem{Gaf}
P{.} M{.} Gaffney.
\newblock A special Stokes theorem for complete Riemannian manifolds.
\newblock {\em Ann. Math.} 60 (1954), 140--145.


\bibitem{Gaud}
P{.} Gauduchon.
\newblock Hermitian connections and Dirac operators.
\newblock {\em Boll. Un. Mat. Ital. B} 11 (1997), no. 2, suppl., 257--288



\bibitem {GoTro}
V{.} Gol'dshtein, M{.} Troyanov.
\newblock The H\"older-Poincar\'e duality for $L_{p,q}$-cohomology
\newblock {\em Ann. Glob. Anal. Geom.} 41 (2012),  25--45.


\bibitem {GoTr}
V{.} Gol'dshtein, M{.} Troyanov.
\newblock Sobolev inequality for differential forms and $L_{q,p}$-cohomology.
\newblock {\em J. Geom. Anal.} 16 (4), 597--631 (2006)

%




\bibitem{GMMI}
C{.} Grant Melles. P{.} Milman.
\newblock Metrics for singular analytic spaces. 
\newblock {\em Pacific J. Math.},  168 (1995), no. 1, 61--156.

\bibitem{GMMIL}
C{.} Grant Melles, P{.} Milman. 
\newblock Classical Poincar\'e metric pulled back off singularities using a Chow type
theorem and desingularization. 
\newblock {\em Ann. Fac. Sci. Toulouse Math.} (6), 15(4):689--771, 2006.

\bibitem{GrRe}
H{.} Grauert, R{.} Remmert.
\newblock Coherent analytic sheaves.
\newblock Grundlehren der MathematischenWissenschaften, vol. 265. Springer, Berlin (1984)


\bibitem{GL}
D{.} Grieser, M{.} Lesch. 
\newblock On the $L^2$-Stokes theorem and Hodge theory for singular algebraic varieties.
\newblock  {\em Math. Nachr.},  246/247 (2002), 68--82.




\bibitem{GYA}
A{.}  Grigor'yan.
\newblock Heat kernel and analysis on manifolds. 
\newblock AMS/IP Studies in Advanced Mathematics, 47. American Mathematical Society, Providence, RI; International Press, Boston, MA, 2009.


\bibitem{GrMa} 
A{.} Grigor'yan, J{.} Masamune.
\newblock Parabolicity and stochastic completeness of manifolds in terms of the Green formula.
\newblock {\em J. Math. Pures Appl.} 100 (5), 607--632 (2013).






\bibitem{Grom}
M{.} Gromov.
\newblock K\"ahler hyperbolicity and $L_2$-Hodge theory. 
\newblock {J. Differential Geom.} 33(1), 263--292 (1991).




\bibitem{masa}
S{.} Haeseler, D{.} Lenz, M{.} Keller, J{.} Masamune, S{.} Schmidt.
\newblock Global properties of Dirichlet forms in terms of Green's formula.
\newblock {\em Calculus of Variations and PDEs} 56 (2017), no. 5.




%

%

\bibitem{Hiro}
H{.} Hironaka.
\newblock Resolution of singularities of an algebraic variety over a field of
characteristic zero, I, II.
\newblock {\em Ann. of Math.} 79 (1964),  109--326.



\bibitem{KiWoo}
F{.} Kirwan, J{.} Woolf.
\newblock An Introduction to Intersection Homology Theory, 2nd edn.
\newblock (Chapman Hall/CRC, 2006).









\bibitem{Loj}
S{.} Lojasiewicz.
\newblock Triangulation of semi-analytic sets. 
\newblock {\em Ann. Scuola Norm. Sup. Pisa} (3) 18, 1964, 449--474. 



%
%
%










\bibitem{TOh}
T{.} Ohsawa. 
\newblock Cheeger-Goreski-MacPherson's conjecture for the varieties with isolated singularities.
\newblock {\em Math. Z.},  206 (1991), no. 2, 219--224. 


\bibitem{ToH} 
T{.} Ohsawa. 
\newblock $L^2$ approaches in several complex variables. Development of Oka-Cartan theory by $L^2$ estimates for the $\overline{\partial}$ operator.
\newblock Springer Monographs in Mathematics. Springer, Tokyo, 2015. 



%
%





\bibitem{PiRiSe}
S{.} Pigola, M{.} Rigoli, A{.} G{.} Setti.
\newblock Vanishing and finiteness results in geometric analysis. A generalization of the Bochner technique.
\newblock Progress in Mathematics, 266. Birkh\"auser Verlag, Basel, 2008.



\bibitem{Rama}
S{.} Ramanan.
\newblock Global calculus. 
\newblock Graduate Studies in Mathematics, 65. American Mathematical Society, Providence, RI, 2005. xii+316 pp. 



%
%
%

\bibitem{JRU}
J{.} Ruppenthal.
\newblock Parabolicity of  the regular locus of complex varieties
\newblock {\em Proc. Amer. Math.} Soc. 144 (2016), no. 1, 225--233.



\bibitem{LeSa}
L{.} Saper.
\newblock $L^2$-cohomology of K\"ahler varieties with isolated singularities.
\newblock {\em J. Differential Geom.}, 31(1): 89--161, 1992.



\bibitem{SaSt}
L{.} Saper, M{.} A{.} Stern.
\newblock $L^2$-cohomology of arithmetic varieties.
\newblock {\em Ann. of Math.} (2), 132 1990, no. 1 , 1--69.


\bibitem{Sella}
Y{.} Sella
\newblock Comparison of sheaf cohomology and singular cohomology.
\newblock https://arxiv.org/abs/1602.06674



\bibitem{ToWeYa}
V{.} Tosatti, B{.} Weinkove, S{.}T{.} Yau.
\newblock Taming symplectic forms and the Calabi--Yau equation.
\newblock {\em Proc. Lond. Math. Soc.} (3) 97 (2008), no. 2, 401--424




\bibitem{paraTroya}
M{.} Troyanov.
\newblock Parabolicity of manifolds.
\newblock {\em Sib. Adv. Math.} 9(4), 125--150 (1999).
 
\bibitem{Val}
G{.} Valette.
\newblock $L^{\infty}$ cohomology is intersection cohomology. (English summary) 
\newblock {\em Adv. Math.} 231 (2012), no. 3--4, 1818--1842.








\bibitem{BoY}
B{.} Youssin.
\newblock $L^p$-cohomology of cones and horns.
\newblock {\em J. Differ. Geom.} 39 (1994) 559--603.


\bibitem{SZu}
S{.} Zucker.
\newblock Hodge theory with degenerating coefficients. $L^2$ cohomology in the Poincar\'e metric.
\newblock {\em Ann. of Math.} (2), 109(3):415--476, 1979.



\end{thebibliography}
\end{document}